\documentclass[10pt,draft,reqno]{amsart}
\usepackage[latin9]{inputenc}
\usepackage{color}
\usepackage{textcomp}
\usepackage{amsmath}
\usepackage{amsthm}
\usepackage{amssymb}
\usepackage{esint}

\makeatletter

\newcommand{\noun}[1]{\textsc{#1}}

\numberwithin{equation}{section} 
\numberwithin{figure}{section} 
\numberwithin{equation}{section}
\numberwithin{figure}{section}
\theoremstyle{plain}
\newtheorem{thm}{\protect\theoremname}[section]
  \theoremstyle{plain}
  \newtheorem{assumption}[thm]{\protect\assumptionname}
  \theoremstyle{remark}
  \newtheorem{rem}[thm]{\protect\remarkname}
  \theoremstyle{plain}
  \newtheorem{lem}[thm]{\protect\lemmaname}
  \theoremstyle{plain}
  \newtheorem{cor}[thm]{\protect\corollaryname}
  \theoremstyle{plain}
  \newtheorem{prop}[thm]{\protect\propositionname}

\usepackage{color}
\usepackage{latexsym}
\usepackage{bm}
\usepackage{graphicx}
\usepackage{wrapfig}
\usepackage{fancybox}

     \def\section{\@startsection{section}{1}%
     \z@{.7\linespacing\@plus\linespacing}{.5\linespacing}%
     {\bfseries
     \centering
     }}
     \def\@secnumfont{\bfseries}

\newcommand{\Rd}{\mathbb{R}^{d}}

\newcommand{\BN}{\mathbb{N}}

\newcommand{\BR}{\mathbb{R}}

\newcommand{\bfE}{\mathbf{E}}

  \providecommand{\corollaryname}{Corollary}
  \providecommand{\lemmaname}{Lemma}
  \providecommand{\propositionname}{Proposition}
  \providecommand{\remarkname}{Remark}
\providecommand{\theoremname}{Theorem}
\providecommand{\assumptionname}{Assumption}

  \providecommand{\corollaryname}{Corollary}
  
  \providecommand{\lemmaname}{Lemma}
  \providecommand{\propositionname}{Proposition}
  \providecommand{\remarkname}{Remark}
\providecommand{\theoremname}{Theorem}

  \providecommand{\corollaryname}{Corollary}
  
  \providecommand{\lemmaname}{Lemma}
  \providecommand{\propositionname}{Proposition}
  \providecommand{\remarkname}{Remark}
\providecommand{\theoremname}{Theorem}

  \providecommand{\corollaryname}{Corollary}
  
  \providecommand{\lemmaname}{Lemma}
  \providecommand{\propositionname}{Proposition}
  \providecommand{\remarkname}{Remark}
\providecommand{\theoremname}{Theorem}

\makeatother

  \providecommand{\assumptionname}{Assumption}
  \providecommand{\corollaryname}{Corollary}
  \providecommand{\lemmaname}{Lemma}
  \providecommand{\propositionname}{Proposition}
  \providecommand{\remarkname}{Remark}
\providecommand{\theoremname}{Theorem}

\begin{document}
\title[Heat kernel estimates for non-symmetric stable-like processes]{Heat kernel estimates for non-symmetric stable-like processes}
\author[P. Jin]{Peng Jin}

\address{Peng Jin: Fakultät für Mathematik und Naturwissenschaften, Bergische
Universität Wuppertal, 42119 Wuppertal, Germany}

\email{jin@uni-wuppertal.de}



\begin{abstract} Let $d\ge1$ and $0<\alpha<2$. Consider the integro-differential
operator
\begin{align*}
\mathcal{L}f(x) & =\int_{\mathbb{R}^{d}\backslash\{0\}}\left[f(x+h)-f(x)-\chi_{\alpha}(h)\nabla f(x)\cdot h\right]\frac{n(x,h)}{|h|^{d+\alpha}}\mathrm{d}h\\
 & \qquad+\mathbf{1}_{\alpha>1}b(x)\cdot\nabla f(x),
\end{align*}
where $\chi_{\alpha}(h):=\mathbf{1}_{\alpha>1}+\mathbf{1}_{\alpha=1}\mathbf{1}_{\{|h|\le1\}}$, $b:\mathbb{R}^{d}\to\mathbb{R}^{d}$ is bounded measurable, and $n:\mathbb{R}^{d}\times\mathbb{R}^{d}\to\mathbb{R}$ is measurable and bounded above and below respectively by two positive constants. Further, we assume that $n(x,h)$ is H\"older continuous in $x$, uniformly with respect to $h\in\mathbb{R}^{d}$. In the case $\alpha=1,$ we assume additionally $\int_{\partial B_{r}}n(x,h)h\mathrm{d}S_{r}(h)=0$, $\forall  r \in (0,\infty)$, where $\mathrm{d}S_{r}$ is the surface measure on $\partial B_{r}$, the boundary of the ball with radius $r$ and center $0$. In this paper, we establish two-sided estimates for the heat kernel of the Markov process associated with the operator $\mathcal{L}$. This extends a recent result of Z.-Q. Chen and X. Zhang. \end{abstract}

\maketitle

\section{Introduction\label{sec:Introduction}}

In probability theory, stable distributions play a very important
role. They appear naturally when one studies the limits of the sum
of suitably rescaled independent and identically distributed random
variables. A stable distribution is firstly characterized by an index
$\alpha\in(0,2]$, which is called the index of stability. Stable
distributions with index $\alpha=2$ are nothing but the Gaussian
ones, while those with index $\alpha\in(0,2)$ have heavy tails and
are particularly interesting for applications, see, e.g., \cite{nolan2016stable}.
One feature of stable distributions is their analytical tractability,
which is due to the simple form of their characteristic functions.
In particular, density estimates for stable distributions with index
$\alpha\in(0,2)$ were done in \cite{MR0270403} for the one-dimensional
case, and the higher dimensional analogues were obtained in \cite{MR0119247,MR1744782,MR2286060}.

A Lévy process whose distribution is $\alpha$-stable is called an
$\alpha$-stable process. Due to \cite{MR1744782,MR2286060}, density
estimates of $\alpha$-stable processes with $\alpha\in(0,2)$ have
been well-understood. Moreover, as shown in \cite{MR2794975,MR3089797,MR3357585},
many other Lévy processes, whose Lévy measure resembles that of an
$\alpha$-stable processes, possess similar or slightly different
density estimates.

Stable-like processes are extensions of stable processes and refer
to Markov processes that behave, at each point of the state space,
like a single stable process. In the literature there are different
definitions of these processes, see, e.g., \cite{MR958291,MR1744782,MR2008600,MR2095633,MR2508568}.
Symmetric stable-like processes can be defined through the corresponding
symmetric Dirichlet forms, as done in \cite{MR2008600}. Note that
sharp heat kernel estimates for symmetric stable-like processes have
been obtained in \cite{MR2008600}. Compared to the symmetric case,
non-symmetric stable-like processes are usually given as solutions
of the martingale problem for stable-like operators. Following \cite{MR2508568},
a stable-like operator $\mathcal{S}$ of order $\alpha\in(0,2)$ takes
the form
\begin{equation}
\mathcal{S}f(x)=\int_{\mathbb{R}^{d}\backslash\{0\}}\left[f(x+h)-f(x)-\mathbf{1}_{\alpha\ge1}\mathbf{1}_{\{|h|\le1\}}\nabla f(x)\cdot h\right]\frac{n(x,h)}{|h|^{d+\alpha}}\mathrm{d}h,\label{defi: S}
\end{equation}
where $f\in C_{b}^{2}(\Rd)$ and the function $n:\Rd\times\Rd\to\mathbb{R}$
are measurable and bounded above and below respectively by two positive
constants. The well-posedness of the martingale problem for $\mathcal{S}$
has been established in \cite{MR2508568,MR3145767,MR3201992,chen2016uniqueness}
under various conditions on $n(x,h)$. It is now known that the stable-like
process corresponding to $\mathcal{S}$ exhibits very similar probabilistic
and analytic properties to a rotationally symmetric $\alpha$-stable
process, see \cite{MR1918242,MR2555009,chen2015heat}; in particular,
its sharp heat kernel estimates have recently been derived in \cite{chen2015heat}
given that $n(x,h)$ is Hölder continuous in $x$ and symmetric in
$h$.

This paper is devoted to the heat kernel estimates of non-symmetric
stable-like processes. We will consider an integro-differential operator
that is more general than the stable-like operator $\mathcal{S}$
given in (\ref{defi: S}). Let $d\ge1$ and $0<\alpha<2$. Consider
the operator
\begin{equation}
\mathcal{L}f(x)=\int_{\mathbb{R}^{d}\backslash\{0\}}\left[f(x+h)-f(x)-\chi_{\alpha}(h)\nabla f(x)\cdot h\right]\frac{n(x,h)}{|h|^{d+\alpha}}\mathrm{d}h+\mathbf{1}_{\alpha>1}b(x)\cdot\nabla f(x),\label{defi: L}
\end{equation}
where $\chi_{\alpha}(h):=\mathbf{1}_{\alpha>1}+\mathbf{1}_{\alpha=1}\mathbf{1}_{\{|h|\le1\}}$,
the vector field $b:\Rd\to\Rd$ and the function $n:\Rd\times\Rd\to\mathbb{R}$
are measurable. Throughout this paper, we assume the following assumptions:
\begin{assumption}
The function $n$ satisfies $0<\kappa_{0}\le n(x,h)\le\kappa_{1}$
for all $x,h\in\Rd$, where $\kappa_{0}$ and $\kappa_{1}$ are constants.
Further, there exist constants $\theta\in(0,1)$ and $\kappa_{2}>0$
such that
\begin{equation}
|n(x,h)-n(y,h)|\le\kappa_{2}|x-y|^{\theta},\quad\forall x,y,h\in\Rd.\label{Condition:Holder}
\end{equation}
In the case $\alpha=1,$ we assume additionally
\begin{equation}
\int_{\partial B_{r}}n(x,h)h\mathrm{d}S_{r}(h)=0,\text{\quad\ensuremath{\forall}}r\text{\ensuremath{\in}}(0,\infty),\label{condition2fornxh_a=00003D00003D1}
\end{equation}
where $dS_{r}$ is the surface measure on $\partial B_{r}$, the boundary
of the ball with center $0$ and radius $r$. \label{ass1:the function n}
\end{assumption}
\begin{rem}
Note that we don't assume the symmetry of $n(x,h)$ in $h$, i.e.,
it is possible that $n(x,h)\neq n(x,-h)$ for some $x,h\in\Rd$.
\end{rem}
\begin{assumption}
There exists a constant $\kappa_{3}>0$ such that $|b(x)|\le\kappa_{3}$
for all $x\in\Rd.$\label{ass2:The-vector-field}
\end{assumption}
According to \cite[Proposition 3]{MR3145767}, the martingale problem
for $\mathcal{L}$ is well-posed under Assumptions \ref{ass1:the function n}
and \ref{ass2:The-vector-field}. In spite of the presence of the
drift term $b\cdot\nabla$ in $\mathcal{L}$, we still call the Markov
process associated with $\mathcal{L}$ a stable-like process. The
main result of this paper is as follows:
\begin{thm}
\label{thm: main}Suppose that the operator $\mathcal{L}$ defined
in \emph{(\ref{defi: L})} satisfies Assumptions \ref{ass1:the function n}
and \ref{ass2:The-vector-field}. Let $\left(X,\left(\mathbf{{L}}^{x}\right)\right)$
be the Markov process associated with $\mathcal{L}$, i.e., $\mathbf{{L}}^{x}$
is the unique solution to the martingale problem for $\mathcal{L}$
starting from $x\in\Rd$ and $X=(X_{t})$ is the canonical process
on $D\big([0,\infty);\Rd\big)$. Then $\left(X,\left(\mathbf{{L}}^{x}\right)\right)$
has a jointly continuous transition density $l(t,x,y)$ such that
$\mathbf{{L}}^{x}\left(X_{t}\in E\right)=\int_{E}l(t,x,y)\mathrm{d}y$
for all $t>0$, $x\in\Rd$ and $E\in\mathcal{B}(\Rd)$. Moreover,
for each $T>0,$ there exists a constant $C_{1}=C_{1}(d,\alpha,\kappa_{0},\kappa_{1},\kappa_{2},\theta,\kappa_{3},T)\in(1,\infty)$
such that
\[
C_{1}^{-1}\left(\frac{t}{|x-y|^{d+\alpha}}\wedge t^{-d/\alpha}\right)\le l(t,x,y)\le C_{1}\left(\frac{t}{|x-y|^{d+\alpha}}\wedge t^{-d/\alpha}\right)
\]
for all $x,y\in\Rd$ and $0<t\le T$. For the case $1<\alpha<2$,
there exists also a constant $C_{2}=C_{2}(d,\alpha,\kappa_{0},\kappa_{1},\kappa_{2},\theta,\kappa_{3},T)>1$
such that
\[
|\nabla_{x}l(t,x,y)|\le C_{2}t^{-1/\alpha}\left(\frac{t}{|x-y|^{d+\alpha}}\wedge t^{-d/\alpha}\right),\quad\forall x,y\in\Rd,\ t\in(0,T].
\]
\end{thm}
To prove Theorem \ref{thm: main}, we will use the same approach as
in \cite{chen2015heat}, namely, we will apply the parametrix method
of Levi. However, we have to overcome two main difficulties. The first
one is, surprisingly, that sharp two-sided density estimates for a
jump-type Lévy process with Lévy measure $K(h)|h|^{-d-\alpha}\mathrm{d}h$,
where $K(\cdot)$ is bounded from above and below by two positive
constants, are not completely known. To solve this problem, we will
start with the upper bounds derived in \cite{MR2794975}, then use
the rescaling argument in \cite[Proposition 2.2]{MR1918242} and some
ideas from \cite{MR2508568} and \cite{MR2095633}. The second difficulty
is due to the fact that $n(x,h)$ is not symmetric in $h$, which
makes some rescaling arguments in \cite{chen2015heat} fail to work.
As a result, in the case $\alpha=1$, we obtain some estimates that
are weaker than those in \cite{chen2015heat} (see, e.g., Lemma \ref{lem:Assume ln}
below and \cite[Theorem 2.4]{chen2015heat}). However, these weaker
forms of estimates don't essentially effect the proof of Theorem \ref{thm: main}.

The rest of the paper is organized as follows. After a short section
on preliminaries, in Section 3 we derive the two-sided density estimates
for jump-type Lévy processes, whose Lévy measure is comparable to
that of a rotationally symmetric $\alpha$-stable process. In Section
4 we construct the transition density of $\left(X,\left(\mathbf{{L}}^{x}\right)\right)$,
with the additional assumption that the drift $b$ in $\mathcal{L}$
is identically $0$. In Section 5 we treat the case where $1<\alpha<2$
and the drift term $b\cdot\nabla$ in $\mathcal{L}$ is present. Section
6 is devoted to the proof of Theorem \ref{thm: main}.

Finally, we give a few remarks on the notation for the constants appearing
in the statements or proofs of the results. The letter $c$ with subscripts
will only appear in proofs and denote positive constants whose exact
value is unimportant. The labeling of the constants $c_{1}$, $c_{2}$,
$...$ starts anew in the proof of each result. We write $C(d,\alpha,...)$
for a positive constant $C$ that depends only on the parameters $d,\alpha,....$

\section{Preliminaries}

\subsection{Notation }

The inner product of $x$ and $y$ in $\mathbb{R}^{d}$ is written
as $x\cdot y$. We use $|v|$ to denote the Euclidean norm of a vector
$v\in\BR^{m}$, $m\in\BN$. We use $B_{r}(x)$ for the open ball of
radius $r$ with center $x$ and simply write $B_{r}$ for $B_{r}(0)$.
The boundary of $B_{r}(x)$ is denoted by $\partial B_{r}(x)$.

For a bounded function $g$ on $\Rd$ we write $\|g\|:=\sup_{x\in\Rd}|g(x)|$.
Let $C_{b}^{2}(\mathbb{R}^{d})$ denote the class of $C^{2}$ functions
such that the function and its first and second order partial derivatives
are bounded.

Let $D=D\big([0,\infty);\Rd\big)$, the set of paths in $\mathbb{R}^{d}$
that are right continuous with left limits, be endowed with the Skorokhod
topology. Set $X_{t}(\omega)=\omega(t)$ for $\omega\in D$ and let
$\mathcal{D}=\sigma(X_{t}:0\le t<\infty)$ and $\mathcal{F}_{t}:=\sigma(X_{r}:0\le r\le t)$.
A probability measure $\mathbf{P}$ on $(D,\mathcal{D})$ is called
a solution to the martingale problem for $\mathcal{L}$ starting from
$x\in\Rd$, if $\mathbf{P}(X_{0}=x)=1$ and under the measure $\mathbf{P}$,
$f(X_{t})-\int_{0}^{t}\mathcal{L}f(X_{u})\mathrm{d}u,\ t\ge0,$ is
an $\mathcal{F}_{t}$-martingale for all $f\in C_{b}^{2}(\mathbb{R}^{d})$.

\subsection{\noun{Rescaling} }

Instead of $\mathcal{L},$ we first consider the operator
\begin{equation}
\mathcal{A}f(x):=\int_{\mathbb{R}^{d}\backslash\{0\}}\left[f(x+h)-f(x)-\chi_{\alpha}(h)h\cdot\nabla f(x)\right]\frac{n(x,h)}{|h|^{d+\alpha}}\mathrm{d}h.\label{defi: A}
\end{equation}
It turns out that the the Markov process associated with $\mathcal{A}$
has the following rescaling property, which is analog to \cite[Proposition 2.2]{MR1918242}.
\begin{lem}
\label{claim:(Rescaling)} Consider the operator $\mathcal{A}$ defined
in \emph{(\ref{defi: A})} with $n(\cdot,\cdot)$ satisfying Assumption
\ref{ass1:the function n}. Let $\left(X,\left(\mathbf{{P}}^{x}\right)\right)$
be the Markov process associated with the operator $\mathcal{A}$,
i.e., $\mathbf{{P}}^{x}$ is the unique solution to the martingale
problem for $\mathcal{A}$ starting from $x\in\Rd$ and $X=(X_{t})$
is the canonical process on $D\big([0,\infty);\Rd\big)$. Let $a>0$.
Define $\mathbf{\tilde{P}}^{x}=\mathbf{P}^{x/a}$ and $Y_{t}:=aX_{a^{-\alpha}t}$,
$t\ge0.$ Then $\mathbf{{\tilde{P}}}^{x}(Y_{0}=x)=1$ and $f(Y_{t})-\int_{0}^{t}\mathcal{\tilde{A}}f(Y_{u})\mathrm{d}u,\ t\ge0,$
is a $\mathbf{\tilde{P}}^{x}$-martingale for all $f\in C_{b}^{2}(\mathbb{R}^{d})$,
where
\begin{align*}
\tilde{\mathcal{A}}f(x) & :=\int_{\mathbb{R}^{d}\backslash\{0\}}\left[f(x+h)-f(x)-\chi_{\alpha}(h)h\cdot\nabla f(x)\right]\frac{\tilde{n}(x,h)}{|h|^{d+\alpha}}\mathrm{d}h
\end{align*}
with $\tilde{n}(x,h):=n(x/a,h/a)$.
\end{lem}
\begin{proof}
In view of (\ref{condition2fornxh_a=00003D00003D1}), the proof of
\cite[Proposition 2.2]{MR1918242} works also here without any changes.
\end{proof}
\begin{rem}
\label{rem: invariance of holder}In Lemma \ref{claim:(Rescaling)},
after the transformation $\tilde{n}(x,h)=n(x/a,h/a)$, we have
\[
|\tilde{n}(x,h)-\tilde{n}(y,h)|=\left|n\left(\frac{x}{a},\frac{h}{a}\right)-n\left(\frac{y}{a},\frac{h}{a}\right)\right|\le\kappa_{2}\left|\frac{x}{a}-\frac{y}{a}\right|^{\theta}=\kappa_{2}a^{-\theta}|x-y|^{\theta}
\]
for all $x,y$ and $h\in\Rd$.
\end{rem}

\subsection{Estimate of the first exit time from a ball}
\begin{lem}
\label{lem: esti exit time of a ball}Let $\mathcal{A}$ and $\left(X,\left(\mathbf{{P}}^{x}\right)\right)$
be as in Lemma \ref{claim:(Rescaling)}. Then there exists a constant
$C_{3}>0$ not depending on $x$ such that for all $r>0$ and $t>0$,
\[
\mathbf{\mathbf{{P}}}^{x}\left(\tau_{B_{r}(x)}\le t\right)\le C_{3}tr{}^{-\alpha},
\]
where $\tau_{B_{r}(x)}:=\inf\left\{ t\ge0:\ X_{t}\notin B_{r}(x)\right\} $.
\end{lem}
\begin{proof}
The proof is essentially identical to that of \cite[Proposition 3.1]{MR2095633}.
Let $f\in C_{b}^{2}(\Rd)$ be a non-negative function that is equal
to $|x|^{2}$ for $|x|\le1/2$, which equals $1$ for $|x|\ge1$.
Let $r>0$ and $x_{0}\in\Rd$ be arbitrary. Define $u(x):=r^{2}f\left(r^{-1}(x-x_{0})\right)$,
$x\in\Rd$. Then $u\in C_{b}^{2}(\Rd)$, and $\|u\|\le c_{1}r^{2}$,
$\|\nabla u\|\le c_{1}r$ and $\|D^{2}u\|\le c_{1}$ for some positive
constant $c_{1}$. As shown in the proof of \cite[Proposition 3.1]{MR2095633},
there exists a constant $c_{2}>0$ such that
\begin{equation}
\left|\int_{{|h|\le r}}\left[u(x+h)-u(x)-h\cdot\nabla u(x)\right]\frac{n(x,h)}{|h|^{d+\alpha}}\mathrm{d}h\right|\le c_{2}r^{2-\alpha}\label{Bass esti 1}
\end{equation}
and
\begin{equation}
\left|\int_{{|h|>r}}\left[u(x+h)-u(x)\right]\frac{n(x,h)}{|h|^{d+\alpha}}\mathrm{d}h\right|\le c_{2}r^{2-\alpha}.\label{Bass esti 2}
\end{equation}
We now distinguish between the following three cases:

(i) $1<\alpha<2$. Since
\[
\left|\int_{{|h|>r}}h\cdot\nabla u(x)\frac{n(x,h)}{|h|^{d+\alpha}}\mathrm{d}h\right|\le c_{1}r\int_{{|h|>r}}\frac{n(x,h)}{|h|^{d+\alpha-1}}\mathrm{d}h\le c_{3}r^{2-\alpha},
\]
we get from (\ref{Bass esti 1}) and (\ref{Bass esti 2}) that $\|\mathcal{A}u\|\le c_{4}r^{2-\alpha}$.

(ii) $\alpha=1$. In view of (\ref{condition2fornxh_a=00003D00003D1}),
it follows directly from (\ref{Bass esti 1}) and (\ref{Bass esti 2})
that $\|\mathcal{A}u\|\le c_{2}r^{2-\alpha}$.

(iii) $0<\alpha<1$. We have
\begin{align*}
\left|\int_{{|h|\le r}}\left[u(x+h)-u(x)\right]\frac{n(x,h)}{|h|^{d+\alpha}}\mathrm{d}h\right| & \le\|\nabla u\|\int_{{|h|\le r}}\frac{n(x,h)}{|h|^{d+\alpha-1}}\mathrm{d}h\le c_{5}r^{2-\alpha},
\end{align*}
which together with (\ref{Bass esti 2}) implies $\|\mathcal{A}u\|\le c_{6}r^{2-\alpha}$.

Further, it was shown in \cite[Proposition 3.1]{MR2095633} that
\begin{align}
r^{2}\mathbf{\mathbf{{P}}}^{x_{0}}\left(\tau_{B_{r}(x_{0})}\le t\right) & \le\mathbf{\mathbf{{E}}}^{x_{0}}\left[u\left(X_{t\wedge\tau_{B_{r}(x_{0})}}\right)\right]\nonumber \\
 & =\mathbf{\mathbf{{E}}}^{x_{0}}\left[\int_{0}^{t\wedge\tau_{B_{r}(x_{0})}}\mathcal{A}u(X_{s})\mathrm{d}s\right]\le c_{7}tr^{2-\alpha},\label{eq 1: Bass}
\end{align}
 which implies the assertion.
\end{proof}
\begin{lem}
\label{lem 2: esti exit time of a ball}Assume $1<\alpha<2$. Let
$\mathcal{L}$ and $\left(X,\left(\mathbf{{L}}^{x}\right)\right)$
be as in Theorem \ref{thm: main}. Define $\tau_{B_{r}(x)}$ as in
Lemma \ref{lem: esti exit time of a ball}. Then for each $T>0$,
there exists a constant $C_{4}>0$ not depending on $x$ such that
for all $0<r<T$ and $t>0$,
\begin{equation}
\mathbf{\mathbf{{L}}}^{x}\left(\tau_{B_{r}(x)}\le t\right)\le C_{4}tr{}^{-\alpha}.\label{eq 2: Bass}
\end{equation}
\end{lem}
\begin{proof}
Let the function $u$ be as in the proof of Lemma \ref{lem: esti exit time of a ball}.
Note that $\mathcal{L}u=\mathcal{A}u+b\cdot\nabla u$ and $\|\mathcal{A}u\|\le c_{1}r^{2-\alpha}$,
$r>0$, which was already proved in proof of Lemma \ref{lem: esti exit time of a ball}.
Then we obtain from $\|b\cdot\nabla u\|\le c_{2}\kappa_{2}r$ that
$\|\mathcal{L}u\|\le c_{3}(r^{2-\alpha}+r)$, $r>0$. Similarly to
(\ref{eq 1: Bass}), we get
\[
r^{2}\mathbf{\mathbf{{L}}}^{x_{0}}\left(\tau_{B_{r}(x_{0})}\le t\right)\le c_{4}t(r^{2-\alpha}+r)\le c_{5}tr^{2-\alpha},\quad0<r<T.
\]
So (\ref{eq 2: Bass}) follows.
\end{proof}

\subsection{Some inequalities and estimates\label{subsec:Some-inequalities-and}}

Let $\gamma>0$ be a constant. It follows from \cite[p.277, (2.9)]{chen2015heat}
that for $|z|\le(2t^{1/\alpha})\lor\left(|x|/2\right)$,
\begin{equation}
\left(t^{1/\alpha}+|x+z|\right)^{-\gamma}\le4^{\gamma}\left(t^{1/\alpha}+|x|\right)^{-\gamma}.\label{ineq 1: chen}
\end{equation}
Following the notation in \cite{chen2015heat}, we write
\[
\varrho_{\gamma}^{\beta}(t,x):=t^{\gamma/\alpha}(|x|^{\beta}\wedge1)(t^{1/\alpha}+|x|)^{-d-\alpha},\quad(t,x)\in(0,\infty)\times\Rd.
\]

As shown in \cite{chen2015heat}, the following convolution inequalities
hold.
\begin{lem}
\emph{(\cite[Lemma 2.1]{chen2015heat})}\noun{ (}\emph{i}\noun{)}
For all $\beta\in[0,\alpha/2]$ and $\gamma\in\mathbb{R}$, there
exists some constant $C_{5}=C_{5}(d,\alpha)>0$ such that
\begin{equation}
\int_{\Rd}\varrho_{\gamma}^{\beta}(t,x)\mathrm{d}x\le C_{5}t^{\frac{\gamma+\beta-\alpha}{\alpha}},\quad(t,x)\in(0,1]\times\Rd.\label{esti1:rho}
\end{equation}
\noun{(}\emph{ii}\noun{)} For all $\beta_{1},\beta_{2}\in[0,\alpha/4]$,
and $\gamma_{1},\gamma_{2}\in\mathbb{R}$, there exists some constant
$C_{6}=C_{6}(d,\alpha)>0$ such that for all $0<s<t\le1$ and $x\in\Rd$,
\begin{align}
 & \int_{\Rd}\varrho_{\gamma_{1}}^{\beta_{1}}(t-s,x-z)\varrho_{\gamma_{2}}^{\beta_{2}}(s,z)\mathrm{d}z\nonumber \\
 & \quad\le C_{6}\left((t-s)^{\frac{\gamma_{1}+\beta_{1}+\beta_{2}-\alpha}{\alpha}}s^{\frac{\gamma_{2}}{\alpha}}+(t-s)^{\frac{\gamma_{1}}{\alpha}}s^{\frac{\gamma_{2}+\beta_{1}+\beta_{2}-\alpha}{\alpha}}\right)\varrho_{0}^{0}(t,x)\nonumber \\
 & \qquad+C_{6}(t-s)^{\frac{\gamma_{1}+\beta_{1}-\alpha}{\alpha}}s^{\frac{\gamma_{2}}{\alpha}}\varrho_{0}^{\beta_{2}}(t,x)+C_{6}(t-s)^{\frac{\gamma_{1}}{\alpha}}s{}^{\frac{\gamma_{2}+\beta_{2}-\alpha}{\alpha}}\varrho_{0}^{\beta_{1}}(t,x).\label{esti2:rho}
\end{align}
\noun{(}\emph{iii}\noun{)} For all $\beta_{1},\beta_{2}\in[0,\alpha/4]$,
$\gamma_{1}+\beta_{1}>0$ and $\gamma_{2}+\beta_{2}>0$, there exists
some constant $C_{7}=C_{7}(d,\alpha)>0$ such that for all $0<s<t\le1$
and $x\in\Rd$,
\begin{align}
 & \int_{0}^{t}\int_{\Rd}\varrho_{\gamma_{1}}^{\beta_{1}}(t-s,x-z)\varrho_{\gamma_{2}}^{\beta_{2}}(s,z)\mathrm{d}z\mathrm{d}s\nonumber \\
 & \quad\le C_{7}\mathcal{B}\left(\frac{\gamma_{1}+\beta_{1}}{\alpha},\frac{\gamma_{2}+\beta_{2}}{\alpha}\right)\left(\varrho_{\gamma_{1}+\gamma_{2}+\beta_{1}+\beta_{2}}^{0}+\varrho_{\gamma_{1}+\gamma_{2}+\beta_{2}}^{\beta_{1}}+\varrho_{\gamma_{1}+\gamma_{2}+\beta_{1}}^{\beta_{2}}\right)(t,x),\label{esti3:rho}
\end{align}
where $\mathcal{B}(\gamma,\beta)$ is the Beta function with parameters
$\gamma,\beta>0$.
\end{lem}
For $\lambda>0$, define $u_{\lambda}(x):=\int_{0}^{\infty}e^{-\lambda t}\varrho_{\alpha}^{0}(s,x)\mathrm{d}x,\ x\in\Rd$.
According to \cite[Lemma 3, Lemma 7 and Theorem 8]{MR2283957}, there
exist constants $C_{8}=C_{8}(d,\alpha)>1$ and $C_{9}=C_{9}(d,\alpha)>1$
such that for all $\lambda>0$ and $x,y,z\in\Rd$,
\begin{align}
 & C_{8}^{-1}\left(\lambda^{(d-\alpha)/\alpha}\lor|x|^{\alpha-d}\right)\wedge\left(\lambda^{-2}|x|^{-d-\alpha}\right)\nonumber \\
 & \qquad\quad\le u_{\lambda}(x)\le C_{8}\left(\lambda^{(d-\alpha)/\alpha}\lor|x|^{\alpha-d}\right)\wedge\left(\lambda^{-2}|x|^{-d-\alpha}\right)\label{esti: u_lambda}
\end{align}
and
\begin{equation}
u_{\lambda}(x-z)\wedge u_{\lambda}(z-y)\le C_{9}u_{\lambda}(x-y).\label{ineq: 3U}
\end{equation}

\begin{lem}
Assume $1<\alpha<2$. Define $k_{\lambda}(x):=\int_{0}^{\infty}e^{-\lambda t}\varrho_{\alpha-1}^{0}(s,x)\mathrm{d}x,\ x\in\Rd.$\emph{
Then there exist constants $C_{10}=C_{10}(d,\alpha)>0$ and $C_{11}=C_{11}(d,\alpha)>0$
such that
\begin{equation}
k_{\lambda}(x)\le C_{10}\left(|x|^{\alpha-d-1}\right)\wedge\left(\lambda^{-2+1/\alpha}|x|^{-d-\alpha}\right),\quad\lambda>0,x\in\Rd,\label{esti: k_lambda}
\end{equation}
and
\begin{equation}
\int_{\Rd}u_{\lambda}(x-z)k_{\lambda}(z-y)\mathrm{d}z\le C_{11}\lambda^{-1+1/\alpha}u_{\lambda}(x-y),\quad\lambda>0,\ x,y\in\Rd.\label{ineq: uk_lambda}
\end{equation}
 \label{Lemma uk_lambda}}
\end{lem}
\begin{proof}
It is easy to see that $k_{\lambda}(x)=\lambda^{(d+1-\alpha)/\alpha}k_{1}\left(\lambda^{1/\alpha}x\right).$
So it suffices to show (\ref{esti: k_lambda}) for $\lambda=1$. For
$x\in\Rd$, we have
\[
k_{1}(x)\le\int_{0}^{|x|^{\alpha}}\frac{e^{-t}t^{1-1/\alpha}}{|x|^{d+\alpha}}\mathrm{d}t+\int_{|x|^{\alpha}}^{\infty}\frac{e^{-t}t^{1-1/\alpha}}{t^{(d+\alpha)/\alpha}}\mathrm{d}t.
\]
Therefore, for $|x|>1$,
\begin{align*}
k_{1}(x) & \le c_{1}|x|^{-d-\alpha}+|x|^{-d-1}\int_{|x|^{\alpha}}^{\infty}e^{-t}\mathrm{d}t\\
 & \le c_{1}|x|^{-d-\alpha}+|x|^{-d-1}e^{-|x|^{\alpha}}\le c_{2}|x|^{-d-\alpha};
\end{align*}
for $|x|\le1$,
\[
k_{1}(x)\le|x|^{-d-\alpha}\int_{0}^{|x|^{\alpha}}t^{1-1/\alpha}\mathrm{d}t+\int_{|x|^{\alpha}}^{\infty}t^{-(d+1)/\alpha}\mathrm{d}t\le c_{3}|x|^{-d+\alpha-1}.
\]

So (\ref{esti: k_lambda}) is true. To show (\ref{ineq: uk_lambda}),
we proceed in the same way as in the proof of \cite[Lemma 17]{MR2283957}.
Set $w_{\lambda}(x):=\left[\left(\lambda^{-(d-\alpha)/\alpha}|x|^{\alpha-d-1}\right)\wedge\left(|x|^{-1}\right)\right]\lor\left(\lambda^{1/\alpha}\right)$.
It follows from (\ref{esti: u_lambda}) and (\ref{esti: k_lambda})
that $k_{\lambda}(x)\le c_{4}w_{\lambda}(x)u_{\lambda}(x)$ for all
$\lambda>0$ and $x\in\Rd$. So
\begin{align}
 & \int_{\Rd}u_{\lambda}(x-z)k_{\lambda}(z-y)\mathrm{d}z\nonumber \\
 & \quad\overset{(\ref{ineq: 3U})}{\le}c_{4}u_{\lambda}(x-y)\int_{\Rd}w_{\lambda}(z-y)\frac{u_{\lambda}(x-z)u_{\lambda}(z-y)}{u_{\lambda}(x-y)}\mathrm{d}z\nonumber \\
 & \quad\le c_{4}u_{\lambda}(x-y)\int_{\Rd}w_{\lambda}(z-y)\left(u_{\lambda}(x-z)\lor u_{\lambda}(z-y)\right)\mathrm{d}z\nonumber \\
 & \quad\le c_{4}u_{\lambda}(x-y)\int_{\Rd}\left[\left(w_{\lambda}(x-z)u_{\lambda}(x-z)\right)\lor\left(w_{\lambda}(z-y)u_{\lambda}(z-y)\right)\right]\mathrm{d}z\label{lemma, uk_lambda: eq0.5}\\
 & \quad\le c_{4}u_{\lambda}(x-y)\int_{\Rd}\left[\left(w_{\lambda}(x-z)u_{\lambda}(x-z)\right)+\left(w_{\lambda}(z-y)u_{\lambda}(z-y)\right)\right]\mathrm{d}z\nonumber \\
 & \quad\le2c_{4}u_{\lambda}(x-y)\int_{\Rd}w_{\lambda}(z)u_{\lambda}(z)\mathrm{d}z,\label{lemma, uk_lambda: eq1}
\end{align}
where in (\ref{lemma, uk_lambda: eq0.5}) we used the fact that $w(z-y)$
and $u_{\lambda}(z-y)$ are decreasing in $|z-y|$. By (\ref{esti: u_lambda})
and the definition of $w_{\lambda}$, we have
\[
w_{\lambda}(z)u_{\lambda}(z)\le c_{5}\left(|z|^{\alpha-d-1}\right)\wedge\left(\lambda^{-2+1/\alpha}|z|^{-d-\alpha}\right),\quad\lambda>0,z\in\Rd.
\]
Thus
\begin{align}
\int_{\Rd}w_{\lambda}(z)u_{\lambda}(z)\mathrm{d}z & \le c_{5}\int_{|z|\le\lambda^{-1/\alpha}}|z|^{\alpha-d-1}\mathrm{d}z+c_{5}\int_{|z|\le\lambda^{-1/\alpha}}\lambda^{-2+1/\alpha}|z|^{-d-\alpha}\mathrm{d}z\nonumber \\
 & \le c_{6}\lambda^{-1+1/\alpha}.\label{lemma, uk_lambda: eq2}
\end{align}
So (\ref{ineq: uk_lambda}) follows by (\ref{lemma, uk_lambda: eq1})
and (\ref{lemma, uk_lambda: eq2}).
\end{proof}

\section{Stable-like Lévy processes and their density estimates}

Consider a Lévy process $Z=(Z_{t})_{t\geq0}$ with $Z_{0}=0$ a.s.,
which is defined on some probability space $(\Omega,\mathcal{A},\mathbf{P})$
and whose characteristic function is given by
\begin{align*}
\bfE\big[e^{iZ_{t}\cdot u}\big] & =e^{-t\psi(u)},\quad u\in\Rd,\\
\psi(u) & =-\int_{\mathbb{R}^{d}\setminus\{0\}}\Big(e^{iu\cdot h}-1-\chi_{\alpha}(h)iu\cdot h\Big)K(h)\mathrm{d}h.
\end{align*}

Throughout this section we assume that the function $K:\Rd\to\mathbb{R}$
satisfies

\begin{equation}
\frac{\kappa_{0}}{|h|^{d+\alpha}}\le K(h)\le\frac{\kappa_{1}}{|h|^{d+\alpha}},\quad h\in\mathbb{R}^{d},\label{condition1forK}
\end{equation}
where $\kappa_{1}>\kappa_{0}>0$ are the constants appearing in Assumption
\ref{ass1:the function n}. In the case $\alpha=1,$ we assume in
addition to (\ref{condition1forK}) that
\begin{equation}
\int_{\partial B_{r}}K(h)z\mathrm{d}S_{r}(h)=0,\text{\quad\ensuremath{\forall}}r\text{\ensuremath{\in}}(0,\infty).\label{condition2forK_a=00003D00003D1}
\end{equation}

In view of (\ref{condition1forK}), we call $Z$ a stable-like Lévy
process. The aim of this section is to establish some estimates for
the density functions of $Z$. To this end, we follow the same idea
as in \cite{chen2015heat}. Define $\tilde{K}:\Rd\to\mathbb{R}$ by
$\tilde{K}(h):=K(h)-\kappa_{0}/(2|h|^{d+\alpha})$, $z\in\Rd$. So
\begin{equation}
\frac{2^{-1}\kappa_{0}}{|h|^{d+\alpha}}\le\tilde{K}(h)\le\frac{\kappa_{1}-2^{-1}\kappa_{0}}{|h|^{d+\alpha}},\quad h\in\mathbb{R}^{d}.\label{bound:K_1}
\end{equation}
Note that if $\alpha=1,$ then
\begin{equation}
\int_{\partial B_{r}}\tilde{K}(h)z\mathrm{d}S_{r}(h)=0,\text{\quad\ensuremath{\forall}}r\text{\ensuremath{\in}}(0,\infty).\label{conditon 2: K tilde}
\end{equation}
Let
\begin{equation}
\tilde{\psi}(u):=-\int_{\mathbb{R}^{d}\setminus\{0\}}\Big(e^{iu\cdot h}-1-\chi_{\alpha}(h)iu\cdot h\Big)\tilde{K}(h)\mathrm{d}h,\quad u\in\Rd,\label{defi:psi1}
\end{equation}
and $\tilde{Z}=(\tilde{Z}_{t})_{t\geq0}$ be a stable-like Lévy process
with the characteristic exponent $\tilde{\psi}$. Without loss of
generality, we assume that the process $(\tilde{Z}_{t})$ is also
defined on $(\Omega,\mathcal{A},\mathbf{P})$.

We can write
\begin{align*}
\psi(u) & =-\int_{\mathbb{R}^{d}\setminus\{0\}}\Big(e^{iu\cdot h}-1-\chi_{\alpha}(h)iu\cdot h\Big)\left(\frac{\kappa_{0}}{2|h|^{d+\alpha}}+\tilde{K}(h)\right)\mathrm{d}h\\
 & =C_{12}|u|^{\alpha}+\tilde{\psi}(u),
\end{align*}
where $C_{12}=C_{12}(d,\alpha,\kappa_{0})>0$ is a constant. It holds
\begin{equation}
e^{-t\Re(\psi(u))}=|e^{-t\psi(u)}|=\big|e^{-t\left(C_{12}|u|^{\alpha}+\tilde{\psi}(u)\right)}\big|=e^{-tC_{12}|u|^{\alpha}}\big|e^{-t\tilde{\psi}(u)}\big|\le e^{-tC_{12}|u|^{\alpha}},\label{MPAFLsect31-1}
\end{equation}
where $\Re(x)$ denotes the real part of $x\in\mathbb{C}$. Therefore,
we get
\begin{equation}
\Re(\psi(u))\ge C_{12}|u|^{\alpha},\quad u\in\Rd,\ t\ge0.\label{MPAFLsect31}
\end{equation}

By (\ref{MPAFLsect31-1}) and the inversion formula of Fourier transform,
the law of $Z_{t}$ has a density (with respect to the Lebesgue measure)
$f_{t}\in L^{1}(\mathbb{R}^{d})\cap C_{b}(\mathbb{R}^{d})$ that is
given by
\begin{equation}
f_{t}(x)=\frac{1}{(2\pi)^{d}}\int_{\mathbb{R}^{d}}e^{-iu\cdot x}e^{-t\psi(u)}\mathrm{d}u,\quad x\in\mathbb{R}^{d},\ t>0.\label{MPAFLsect312}
\end{equation}

Similarly, we define
\begin{equation}
g_{t}(x):=\frac{1}{(2\pi)^{d}}\int_{\mathbb{R}^{d}}e^{-iu\cdot x}e^{-tC_{12}|u|^{\alpha}}\mathrm{d}u\label{defi: g_t}
\end{equation}
and
\[
\tilde{f}_{t}(x):=\frac{1}{(2\pi)^{d}}\int_{\mathbb{R}^{d}}e^{-iu\cdot x}e^{-t\tilde{\psi}(u)}\mathrm{d}u
\]
for $x\in\mathbb{R}^{d},\ t>0$. Then $g_{t}$ and $h_{t}$ are densities
of some rotationally symmetric $\alpha$-stable process $(S_{t})$
and the stable-like Lévy process $(\tilde{Z}_{t})$, respectively.
It is clear that $f_{t}=g_{t}*\tilde{f}_{t}$. Since $g_{t}$ is the
density of a rotationally symmetric $\alpha$-stable process, we have
the following scaling property of $g_{t}$: for all $x\in\mathbb{R}^{d}$
and $t>0$,
\begin{equation}
g_{t}(x)=t^{-d/\alpha}g_{1}(t^{-1/\alpha}x).\label{scaling : gt}
\end{equation}
It is well-known that the following estimates for $g_{t}$ hold: there
exists some constant $C_{13}=C_{13}(d,\alpha,\kappa_{0})>1$ such
that
\begin{equation}
C_{13}^{-1}t\left(t^{1/\alpha}+|x|\right)^{-d-\alpha}\le g_{t}(x)\le C_{13}t\left(t^{1/\alpha}+|x|\right)^{-d-\alpha}\label{twosidebound:g_t}
\end{equation}
for all $x\in\Rd$ and $t>0$. Moreover, for each $k\in\mathbb{N}$,
we can find a constant $C_{14}=C_{14}(d,\alpha,\kappa_{0},k)>0$ such
that
\begin{equation}
|\nabla^{k}g_{t}(x)|\le C_{14}t\left(t^{1/\alpha}+|x|\right)^{-d-\alpha-k}\label{bound:gradientg_t}
\end{equation}
for all $x\in\Rd$ and $t>0$, see \cite[Lemma 2.2]{chen2015heat}.

We next show that the same estimate as in (\ref{twosidebound:g_t})
is also true for the density $f_{t}$. For $|\nabla f_{t}|$ we shall
derive an estimate that is slightly worse than the estimate on $|\nabla g_{t}|$
given in (\ref{bound:gradientg_t}). As the first step, we have the
following upper estimate that is actually a special case of \cite[Theorem 1]{MR2794975}.
\begin{lem}
\emph{(\cite{MR2794975}) }Let $f_{t}$ be as in \emph{(\ref{MPAFLsect312})}.
Then there exists some constant $C_{15}=C_{15}(d,\alpha,\kappa_{0},\kappa_{1})>0$
such that
\begin{equation}
f_{1}(x)\le C_{15}\left(1\wedge|x|{}^{-d-\alpha}\right),\quad x\in\Rd.\label{esti f_1}
\end{equation}
\label{MPAFLlemma1}
\end{lem}
\begin{proof}
Note that (\ref{MPAFLsect31}) is true. The assertion thus follows
by \cite[Theorem 1]{MR2794975}. Indeed, to apply \cite[Theorem 1]{MR2794975},
we only need to take $\mu$ as the surface measure $\mathrm{d}S_{1}$
on $\partial B_{1}$, $q(\cdot)\equiv\kappa_{1}$, $\phi(\cdot)\text{\ensuremath{\equiv}}1$,
$\beta=\alpha$, $\gamma=d$, and $k_{1}=k_{2}=1$ there. Then we
obtain
\begin{equation}
f_{1}(x+v)\le c_{1}\left(1\wedge|x|{}^{-d-\alpha}\right),\quad\forall x\in\Rd,\label{esti f_1 old}
\end{equation}
where $c_{1}=c_{1}(d,\alpha,\kappa_{0},\kappa_{1})>0$ is a constant
and the vector $v\in\Rd$ is defined by
\[
v:=\begin{cases}
-\int_{|z|\ge1}hK(h)\mathrm{d}h, & 1<\alpha<2,\\
0, & \alpha=1,\\
\int_{0<|z|<1}hK(h)\mathrm{d}h, & 0<\alpha<1.
\end{cases}
\]
It follows from (\ref{condition1forK}) that $|v|\le c_{2}$, where
$c_{2}=c_{2}(d,\alpha,\kappa_{1})>0$ is a constant. The estimate
(\ref{esti f_1}) now follows from (\ref{esti f_1 old}).
\end{proof}
\begin{lem}
Let $C_{15}$ be as in Lemma \ref{MPAFLlemma1}. Then we have
\begin{equation}
f_{t}(x)\le C_{15}t\left(t^{1/\alpha}+|x|\right)^{-d-\alpha},\quad x\in\Rd,\,t>0.\label{upperbound:f_t}
\end{equation}
Moreover, there exists some constant $C_{16}=C_{16}(d,\alpha,\kappa_{0},\kappa_{1})>0$
such that
\begin{equation}
|\nabla f_{t}(x)|\le C_{16}t^{1-1/\alpha}\left(t^{1/\alpha}+|x|\right)^{-d-\alpha}\label{upperbound:gradientf_t}
\end{equation}
 for all $x\in\Rd$ and $t>0$. \label{lem: upper esti for f_t}
\end{lem}
\begin{proof}
Let $a>0$ and define $Y_{t}:=aZ_{a^{-\alpha}t}$, $t\ge0$. Then
$(Y_{t})$ is a Lévy process and for $u\in\Rd,$
\begin{align*}
\bfE\big[e^{iY_{t}\cdot u}\big] & =\bfE\big[e^{iaL_{a^{-\alpha}t}\cdot u}\big]\\
 & =\exp\left(ta^{-\alpha}\int_{\mathbb{R}^{d}\setminus\{0\}}\Big(e^{iau\cdot h}-1-\chi_{\alpha}(h)iau\cdot h\Big)K(h)\mathrm{d}h\right)
\end{align*}
By (\ref{condition2forK_a=00003D00003D1}) and a change of variables,
we obtain
\[
\bfE\big[e^{iY_{t}\cdot u}\big]=\exp\left(\int_{\mathbb{R}^{d}\setminus\{0\}}\Big(e^{iu\cdot h}-1-\chi_{\alpha}(h)iu\cdot h\Big)a^{-d-\alpha}K(a^{-1}h)\mathrm{d}h\right),\quad u\in\Rd.
\]
Set $M(h):=a^{-d-\alpha}K(a^{-1}h)$, $h\in\Rd$. Then the function
$M$ satisfies

\begin{equation}
\frac{\kappa_{0}}{|h|^{d+\alpha}}\le M(h)\le\frac{\kappa_{1}}{|h|^{d+\alpha}},\quad h\in\mathbb{R}^{d},\label{condition1forM}
\end{equation}
where the positive constants $\kappa_{0}$ and $\kappa_{1}$ are the
same as in (\ref{condition1forK}). Therefore, $(Y_{t})$ is also
a stable-like Lévy process. Let $\rho(x),$ $x\in\Rd,$ be the probability
density of $Y_{1}$. By choosing $a$ such that $a^{-\alpha}=t$,
we obtain $Y_{1}=t^{-1/\alpha}Z_{t}$, which implies $\rho(x)=t^{d/\alpha}f_{t}(t^{1/\alpha}x),\ x\in\Rd.$
It follows from Lemma \ref{MPAFLlemma1} that $t^{d/\alpha}f_{t}(t^{1/\alpha}x)\le C_{15}\left(1\wedge|x|{}^{-d-\alpha}\right),\ x\in\Rd.$
So (\ref{upperbound:f_t}) is true.

Next, we will use the fact that $f_{t}=g_{t}*\tilde{f}_{t}$ to show
(\ref{upperbound:gradientf_t}). Since $\tilde{f}_{t}$ is the density
of $\tilde{L}_{t}$ and $(\tilde{L}_{t})$ is a stable-like Lévy process
with the jump kernel $\tilde{K}$ that satisfies (\ref{bound:K_1})
and (\ref{conditon 2: K tilde}), we obtain, using (\ref{upperbound:f_t}),
the existence of a constant $\tilde{C}_{15}=\tilde{C}_{15}(d,\alpha,\kappa_{0},\kappa_{1})>0$
such that
\begin{equation}
\tilde{f}_{t}(x)\le\tilde{C}_{15}t\left(t^{1/\alpha}+|x|\right)^{-d-\alpha},\quad x\in\Rd,t>0.\label{upperbound:h_t}
\end{equation}
Note that $\nabla f_{t}=(\nabla g_{t})*\tilde{f}_{t}.$ By (\ref{bound:gradientg_t}),
we get that for all $x\in\Rd$ and $t>0$,
\begin{align*}
|\nabla f_{t}(x)| & \le\int_{\Rd}|\nabla g_{t}(x-h)|\tilde{f}_{t}(h)\mathrm{d}h\\
 & \le C_{14}\tilde{C}_{15}\int_{\Rd}t\left(t^{1/\alpha}+|x-h|\right)^{-d-\alpha-1}t\left(t^{1/\alpha}+|h|\right)^{-d-\alpha}\mathrm{d}h\\
 & \le C_{16}t^{1-1/\alpha}\left(t^{1/\alpha}+|x|\right)^{-d-\alpha}.
\end{align*}
This completes the proof.
\end{proof}
By (\ref{upperbound:f_t}) and the same argument as in \cite[Lemma 2.3]{chen2015heat},
we easily obtain the following corollary.
\begin{cor}
There exists a constant $C_{17}=C_{17}(d,\alpha,\kappa_{0},\kappa_{1})>0$
such that
\begin{equation}
|f_{t}(x)-f_{t}(x')|\le C_{17}\left(\left(t^{-1/\alpha}|x-x'|\right)\wedge1\right)\left\{ \varrho_{\alpha}^{0}(t,x)+\varrho_{\alpha}^{0}(t,x')\right\} \label{Coro 1, eq}
\end{equation}
for all $x,x'\in\Rd$ and $t>0$.
\end{cor}
\begin{lem}
There exists some constant $C_{18}=C_{18}(d,\alpha,\kappa_{0},\kappa_{1})>0$
such that
\[
f_{t}(x)\ge C_{18}t\left(t^{1/\alpha}+|x|\right)^{-d-\alpha},\quad\forall x\in\Rd,t>0.
\]
\label{lem: lower esti for f_t}
\end{lem}
\begin{proof}
We will use the fact that $f_{t}=g_{t}*\tilde{f}_{t}$ to show this
lemma. According to Lemma \ref{lem: esti exit time of a ball}, there
exists some constant $c_{1}=c_{1}(d,\alpha,\kappa_{0},\kappa_{1})>0$
such that
\begin{equation}
\mathbf{P}\left(\tilde{\tau}_{B_{r}}\le t\right)\le c_{1}tr{}^{-\alpha},\quad\forall r>0,\label{ineq:uniformboundtau}
\end{equation}
where $\tilde{\tau}_{B_{r}}:=\inf\left\{ t\ge0:\ \tilde{Z}_{t}\notin B_{r}\right\} $.
Choose $c_{2}>0$ such that
\begin{equation}
(2^{-1}c_{2})^{\alpha}=2c_{1}.\label{eq:c_2}
\end{equation}
If $|x|\le c_{2}t^{1/\alpha}$, then
\begin{align*}
f_{t}(x) & \ge\int_{B_{2c_{2}t^{1/\alpha}}(x)}g_{t}(x-y)\tilde{f}_{t}(y)\mathrm{d}y\ge c_{3}t^{-d/\alpha}\int_{B_{2c_{2}t^{1/\alpha}}(x)}\tilde{f}_{t}(y)\mathrm{d}y\\
 & \ge c_{3}t^{-d/\alpha}\int_{B_{c_{2}t^{1/\alpha}}}\tilde{f}_{t}(y)\mathrm{d}y=c_{3}t^{-d/\alpha}\mathbf{P}\left(\tilde{Z}_{t}\in B_{c_{2}t^{1/\alpha}}\right)\\
 & \ge c_{3}t^{-d/\alpha}\mathbf{P}\left(\sup_{0\le s\le t}|\tilde{Z}_{s}|<c_{2}t^{1/\alpha}\right)=c_{3}t^{-d/\alpha}\left(1-\mathbf{P}\left(\sup_{0\le s\le t}|\tilde{Z}_{s}|\ge c_{2}t^{1/\alpha}\right)\right)\\
 & =c_{3}t^{-d/\alpha}\left(1-\mathbf{P}\left(\tilde{\tau}_{B_{c_{2}t^{1/\alpha}}}\le t\right)\right)\overset{(\ref{ineq:uniformboundtau})}{\ge}c_{3}t^{-d/\alpha}\left(1-c_{1}t\left(c_{2}t^{1/\alpha}\right){}^{-\alpha}\right)\\
 & \overset{(\ref{eq:c_2})}{=}c_{4}t^{-d/\alpha}.
\end{align*}
If $|x|>c_{2}t^{1/\alpha}$, then
\begin{align*}
f_{t}(x) & \ge\int_{\Rd\setminus B_{2^{-1}c_{2}t^{1/\alpha}}(x)}g_{t}(x-y)\tilde{f}_{t}(y)\mathrm{d}y\ge c_{5}\int_{\Rd\setminus B_{2^{-1}c_{2}t^{1/\alpha}}(x)}\frac{t}{|x-y|^{d+\alpha}}\tilde{f}_{t}(y)\mathrm{d}y\\
 & \ge c_{5}\int_{\Rd\setminus B_{2^{-1}c_{2}t^{1/\alpha}}(x)}\frac{t}{|x-y|^{d+\alpha}}\tilde{f}_{t}(y)\mathrm{d}y\ge c_{5}\int_{B_{2^{-1}c_{2}t^{1/\alpha}}}\frac{t}{|x-y|^{d+\alpha}}\tilde{f}_{t}(y)\mathrm{d}y\\
 & \ge c_{6}\frac{t}{|x|^{d+\alpha}}\int_{B_{2^{-1}c_{2}t^{1/\alpha}}}\tilde{f}_{t}(y)\mathrm{d}y=c_{6}\frac{t}{|x|^{d+\alpha}}\left(1-\mathbf{P}\left(\tilde{\tau}_{B_{2^{-1}c_{2}t^{1/\alpha}}}\le t\right)\right)\\
 & \overset{(\ref{ineq:uniformboundtau})}{\ge}c_{6}\frac{t}{|x|^{d+\alpha}}\left(1-c_{1}t\left(2^{-1}c_{2}t^{1/\alpha}\right){}^{-\alpha}\right)\overset{(\ref{eq:c_2})}{=}c_{7}\frac{t}{|x|^{d+\alpha}}.
\end{align*}
 This completes the proof.
\end{proof}
Next, we derive some useful estimates for $f_{t}$. In the subsequent
proofs we will use very often the following identities: for $t>0$
and $x,h\in\Rd$,
\begin{equation}
g_{t}(x+h)-g_{t}(x)=\int_{0}^{1}\nabla g_{t}(x+rh)\cdot h\mathrm{d}r,\label{identity 1: frac esti f_t}
\end{equation}

\begin{equation}
g_{t}(x+h)-g_{t}(x)-h\cdot\nabla g_{t}(x)=\int_{0}^{1}\left(\int_{0}^{1}\nabla^{2}g_{t}(x+rr'h)\cdot rh\mathrm{d}r'\right)\cdot h\mathrm{d}r.\label{identity 2: frac esti f_t}
\end{equation}

For each $\alpha\in(0,2)$, it was proved in \cite[p. 282]{chen2015heat}
that there exists some constant $C_{19}=C_{19}(d,\alpha)>0$ such
that for all $0<t\le1$ and $x\in\Rd$,
\begin{equation}
\int_{\Rd}\left(\left(t^{-2/\alpha}|h|^{2}\right)\wedge1\right)\left(\varrho_{\alpha}^{0}(t,x+h)+\varrho_{\alpha}^{0}(t,x)\right)\cdot|h|^{-d-\alpha}\mathrm{d}h\le C_{19}\varrho_{0}^{0}(t,x).\label{Lemma 8, eq 3-1}
\end{equation}

\begin{lem}
Assume $\alpha\neq1$. Then there exists constant $C_{20}=C_{20}(d,\alpha,\kappa_{0},\kappa_{1})>0$
such that for all $0<t\le1$ and $x\in\Rd$,
\begin{equation}
\int_{\Rd}\left|f_{t}(x+h)-f_{t}(x)-\chi_{\alpha}(h)h\cdot\nabla f_{t}(x)\right|\cdot|h|^{-d-\alpha}\mathrm{d}h\le C_{20}\varrho_{0}^{0}(t,x).\label{esti:f_t,integral_differ}
\end{equation}
\label{lem:frac esti f_t}
\end{lem}
\begin{proof}
The idea of proof is borrowed from \cite[Theorem 2.4]{chen2015heat}.
If we can find a constant $\tilde{C}_{20}=\tilde{C}_{20}(d,\alpha,\kappa_{0},\kappa_{1})>0$
such that for all $0<t\le1$ and $x\in\Rd$,
\begin{equation}
\int_{\Rd}\left|g_{t}(x+h)-g_{t}(x)-\chi_{\alpha}(h)h\cdot\nabla g_{t}(x)\right|\cdot|h|^{-d-\alpha}\mathrm{d}h\le\tilde{C}_{20}\varrho_{0}^{0}(t,x),\label{eq aim: frac esti f_t}
\end{equation}
 then the assertion follows from $f_{t}=g_{t}*\tilde{f}_{t}$ and
\begin{align*}
\int_{\Rd}\varrho_{0}^{0}(t,x-y)\tilde{f}_{t}(y)\mathrm{d}y & \le c_{1}t^{-1}\int_{\Rd}\tilde{f}_{t}(x-y)\tilde{f}_{t}(y)\mathrm{d}y\\
 & =c_{1}t^{-1}\tilde{f}_{2t}(x)\le c_{2}\varrho_{0}^{0}(t,x).
\end{align*}
 Next, we proceed to prove (\ref{eq aim: frac esti f_t}).

(i) We first consider the case $0<\alpha<1$. If $|h|\le1$, then
\begin{align*}
|g_{1}(x+h)-g_{1}(x)| & \overset{(\ref{identity 1: frac esti f_t})}{\le}|h|\int_{0}^{1}|\nabla g_{1}(x+rh)|\mathrm{d}r\\
 & \overset{(\ref{bound:gradientg_t})}{\le}c_{3}|h|\int_{0}^{1}(1+|x+rh|)^{-d-\alpha-1}\mathrm{d}r\\
 & \overset{(\ref{ineq 1: chen})}{\le}c_{4}|h|(1+|x|)^{-d-\alpha-1}\le c_{4}|h|(1+|x|)^{-d-\alpha}.
\end{align*}
So
\begin{equation}
|g_{1}(x+h)-g_{1}(x)|\le c_{5}\left(|h|\wedge1\right)\left(\varrho_{\alpha}^{0}(1,x+h)+\varrho_{\alpha}^{0}(1,x)\right).\label{new eq1: Lemma 9}
\end{equation}
By (\ref{scaling : gt}), we get
\[
|g_{t}(x+h)-g_{t}(x)|\le c_{5}\left(\left(t^{-1/\alpha}|h|\right)\wedge1\right)\left(\varrho_{\alpha}^{0}(t,x+h)+\varrho_{\alpha}^{0}(t,x)\right).
\]
Therefore,
\begin{align*}
 & \int_{\mathbb{R}^{d}}|g_{t}(x+h)-g_{t}(x)|\cdot|h|^{-d-\alpha}\mathrm{d}h\\
 & \quad\le c_{5}\int_{\mathbb{R}^{d}}\left(\left(t^{-1/\alpha}|h|\right)\wedge1\right)\varrho_{\alpha}^{0}(t,x+h)\cdot|h|^{-d-\alpha}\mathrm{d}h\\
 & \qquad+c_{5}\int_{\mathbb{R}^{d}}\left(\left(t^{-1/\alpha}|h|\right)\wedge1\right)\varrho_{\alpha}^{0}(t,x)\cdot|h|^{-d-\alpha}\mathrm{d}h=:I_{1}+I_{2}.
\end{align*}
 We have
\begin{align*}
 & I_{1}\le c_{5}t^{-1/\alpha}\int_{|h|\le t^{1/\alpha}}\varrho_{\alpha}^{0}(t,x+h)\cdot|h|^{-d-\alpha+1}\mathrm{d}h\\
 & \qquad+c_{5}\int_{|h|>t^{1/\alpha}}\varrho_{\alpha}^{0}(t,x+h)\cdot|h|^{-d-\alpha}\mathrm{d}h=:I_{11}+I_{12}.
\end{align*}
 Further,
\begin{align*}
 & I_{11}\le c_{5}t^{1-1/\alpha}\int_{|h|\le t^{1/\alpha}}\left(t^{1/\alpha}+|x+h|\right)^{-d-\alpha}\cdot|h|^{-d-\alpha+1}\mathrm{d}h\\
 & \quad\overset{(\ref{ineq 1: chen})}{\le}c_{6}t^{1-1/\alpha}\left(t^{1/\alpha}+|x|\right)^{-d-\alpha}\int_{|h|\le t^{1/\alpha}}|h|^{-d-\alpha+1}\mathrm{d}h\le c_{7}\varrho_{0}^{0}(t,x).
\end{align*}
If $|x|\le2t^{1/\alpha}$, then
\begin{align*}
 & I_{12}\le c_{5}t\int_{|h|>t^{1/\alpha}}\left(t^{1/\alpha}+|x+h|\right)^{-d-\alpha}\cdot|h|^{-d-\alpha}\mathrm{d}h\\
 & \quad\le c_{5}t^{-d/\alpha}\int_{|h|>t^{1/\alpha}}|h|^{-d-\alpha}\mathrm{d}h\le c_{8}t^{-1-d/\alpha}\le c_{9}\varrho_{0}^{0}(t,x);
\end{align*}
if $|x|>2t^{1/\alpha}$, then
\begin{align*}
 & I_{12}\le c_{5}\left(\int_{t^{1/\alpha}<|h|\le\frac{|x|}{2}}+\int_{|h|>\frac{|x|}{2}}\right)\varrho_{\alpha}^{0}(t,x+h)\cdot|h|^{-d-\alpha}\mathrm{d}h\\
 & \quad\le c_{5}t\int_{t^{1/\alpha}<|h|\le\frac{|x|}{2}}\left(t^{1/\alpha}+|x+h|\right)^{-d-\alpha}\cdot|h|^{-d-\alpha}\mathrm{d}h\\
 & \qquad+c_{10}|x|^{-d-\alpha}\int_{|h|>\frac{|x|}{2}}\varrho_{\alpha}^{0}(t,x+h)\mathrm{d}h\\
 & \quad\overset{(\ref{ineq 1: chen})}{\le}c_{11}t\left(t^{1/\alpha}+|x|\right)^{-d-\alpha}\int_{t^{1/\alpha}<|h|\le\frac{|x|}{2}}|h|^{-d-\alpha}\mathrm{d}h+c_{12}|x|^{-d-\alpha}\\
 & \quad\le c_{13}\left(t^{1/\alpha}+|x|\right)^{-d-\alpha}+c_{12}|x|^{-d-\alpha}\le c_{14}\varrho_{0}^{0}(t,x).
\end{align*}
For $I_{2}$, by setting $\tilde{h}:=t^{-1/\alpha}h$, we have
\begin{align*}
 & I_{2}=c_{5}\varrho_{\alpha}^{0}(t,x)\int_{\mathbb{R}^{d}}\left(|\tilde{h}|\wedge1\right)\cdot|t^{1/\alpha}\tilde{h}|^{-d-\alpha}t^{d/\alpha}\mathrm{d}\tilde{h}\\
 & \quad=c_{5}t^{-1}\varrho_{\alpha}^{0}(t,x)\int_{\mathbb{R}^{d}}\left(|\tilde{h}|\wedge1\right)\cdot|\tilde{h}|^{-d-\alpha}\mathrm{d}\tilde{h}\le c_{15}\varrho_{0}^{0}(t,x).
\end{align*}
Summarizing the above estimates for $I_{11}$, $I_{12}$ and $I_{2}$,
we obtain (\ref{eq aim: frac esti f_t}).

(ii) Let $1<\alpha<2$. For $|h|>1$, we have
\begin{align*}
 & \left|g_{1}(x+h)-g_{1}(x)-\chi_{\alpha}(h)h\cdot\nabla g_{1}(x)\right|\\
 & \quad\le g_{1}(x+h)+g_{1}(x)+|h|\cdot|\nabla g_{1}|(x)\\
 & \quad\overset{(\ref{bound:gradientg_t})}{\le}c_{16}\left(\varrho_{\alpha}^{0}(1,x+h)+\varrho_{\alpha}^{0}(1,x)\right)+c_{17}|h|\varrho_{\alpha-1}^{0}(1,x).
\end{align*}
For $|h|\le1$, we have
\begin{align}
 & \left|g_{1}(x+h)-g_{1}(x)-\chi_{\alpha}(h)h\cdot\nabla g_{1}(x)\right|\nonumber \\
 & \quad\overset{(\ref{identity 2: frac esti f_t})}{\le}|h|^{2}\int_{0}^{1}\int_{0}^{1}\left|\nabla^{2}g_{1}(x+rr'h)\right|\mathrm{\mathrm{d}\mathit{r'}d}r\nonumber \\
 & \quad\overset{(\ref{bound:gradientg_t})}{\le}c_{18}|h|^{2}\int_{0}^{1}\int_{0}^{1}(1+|x+rr'h|)^{-d-\alpha-2}\mathrm{\mathrm{d}\mathit{r'}d}r\nonumber \\
 & \quad\overset{(\ref{ineq 1: chen})}{\le}c_{19}|h|^{2}(1+|x|)^{-d-\alpha-2}\le c_{19}|h|^{2}(1+|x|)^{-d-\alpha}.\label{eq0: frac esti f_t}
\end{align}
So
\begin{align}
 & \left|g_{1}(x+h)-g_{1}(x)-\chi_{\alpha}(h)h\cdot\nabla g_{1}(x)\right|\nonumber \\
 & \quad\le c_{20}\left(|h|^{2}\wedge1\right)\left(\varrho_{\alpha}^{0}(1,x+h)+\varrho_{\alpha}^{0}(1,x)\right)+c_{21}\mathbf{1}_{\{|h|>1\}}|h|\varrho_{\alpha-1}^{0}(1,x).\label{Lemma 8, eq1}
\end{align}
By (\ref{scaling : gt}), we get
\begin{align}
 & \left|g_{t}(x+h)-g_{t}(x)-\chi_{\alpha}(h)h\cdot\nabla g_{t}(x)\right|\nonumber \\
 & \quad=t^{-d/\alpha}\left|g_{1}(t^{-1/\alpha}x+t^{-1/\alpha}h)-g_{1}(t^{-1/\alpha}x)-t^{-1/\alpha}h\cdot\nabla g_{1}(t^{-1/\alpha}x)\right|\nonumber \\
 & \quad\overset{(\ref{Lemma 8, eq1})}{\le}c_{20}\left(\left(t^{-2/\alpha}|h|^{2}\right)\wedge1\right)\left(\varrho_{\alpha}^{0}(t,x+h)+\varrho_{\alpha}^{0}(t,x)\right)+c_{21}\mathbf{1}_{\{|h|>t^{1/\alpha}\}}|h|\varrho_{\alpha-1}^{0}(t,x).\label{Lemma 8, eq 2}
\end{align}
Since
\begin{align*}
 & \int_{|h|>t^{1/\alpha}}|h|\varrho_{\alpha-1}^{0}(t,x)|h|^{-d-\alpha}\mathrm{d}h\le c_{22}\varrho_{0}^{0}(t,x),
\end{align*}
the assertion now follows from (\ref{Lemma 8, eq 3-1}) and (\ref{Lemma 8, eq 2}).
\end{proof}
\begin{lem}
\label{lem:Assume ln}Assume $\alpha=1$. Then there exists a constant
$C_{21}=C_{21}(d,\alpha,\kappa_{0},\kappa_{1})>0$ such that for all
$0<t\le1$ and $x\in\Rd$,
\begin{equation}
\int_{\Rd}\lvert f_{t}(x+h)-f_{t}(x)-\chi_{\alpha}(h)h\cdot\nabla f_{t}(x)\rvert\cdot\frac{1}{|h|^{d+\alpha}}\mathrm{d}h\le C_{21}\left(1+\ln\left(t^{-1}\right)\right)\varrho_{0}^{0}(t,x).\label{esti:f_t,integral_differ, a=00003D1}
\end{equation}
\end{lem}
\begin{proof}
Note that $\chi_{\alpha}(h)=\mathbf{1}_{\{|h|\le1\}}$ when $\alpha=1$.
Similarly to (\ref{eq0: frac esti f_t}), we have that for $|h|\le1$,
\[
\left|g_{1}(x+h)-g_{1}(x)-\chi_{\alpha}(h)h\cdot\nabla g_{1}(x)\right|\le c_{1}|h|^{2}(1+|x|)^{-d-\alpha}.
\]
 For $|h|>1$, we have
\[
\left|g_{1}(x+h)-g_{1}(x)-\chi_{\alpha}(h)h\cdot\nabla g_{1}(x)\right|\le c_{2}\left(\varrho_{\alpha}^{0}(1,x+h)+\varrho_{\alpha}^{0}(1,x)\right).
\]
So
\begin{equation}
\left|g_{1}(x+h)-g_{1}(x)-\chi_{\alpha}(h)h\cdot\nabla g_{1}(x)\right|\le c_{3}\left(|h|^{2}\wedge1\right)\left(\varrho_{\alpha}^{0}(1,x+h)+\varrho_{\alpha}^{0}(1,x)\right).\label{eq1.5: frac esti f_t}
\end{equation}
By the scaling property $g_{t}(x)=t^{-d/\alpha}g_{1}(t^{-1/\alpha}x)$,
we obtain
\begin{align}
 & \left|g_{t}(x+h)-g_{t}(x)-\chi_{\alpha}(h)h\cdot\nabla g_{t}(x)\right|\nonumber \\
 & \quad=t^{-d}\left|g_{1}(t^{-1}x+t^{-1}h)-g_{1}(t^{-1}x)-t^{-1}\chi_{1}(h)h\cdot\nabla g_{1}(t^{-1}x)\right|\nonumber \\
 & \quad=t^{-d}\big|g_{1}(t^{-1}x+t^{-1}h)-g_{1}(t^{-1}x)-\chi_{1}(t^{-1}h)t^{-1}h\cdot\nabla g_{1}(t^{-1}x)\nonumber \\
 & \qquad-\mathbf{1}_{\left\{ t<|h|\le1\right\} }(h)t^{-1}h\cdot\nabla g_{1}(t^{-1}x)\big|\nonumber \\
 & \overset{(\ref{eq1.5: frac esti f_t}),(\ref{bound:gradientg_t})}{\le}c_{3}\left(|t^{-1}h|^{2}\wedge1\right)\left(\varrho_{\alpha}^{0}(t,x+h)+\varrho_{\alpha}^{0}(t,x)\right)\nonumber \\
 & \qquad\qquad+c_{4}\mathbf{1}_{\left\{ t<|h|\le1\right\} }(h)t^{-d-1}(1+|t^{-1}x|)^{-d-2}|h|\nonumber \\
 & \quad\le c_{3}\left(|t^{-1}h|^{2}\wedge1\right)\left(\varrho_{\alpha}^{0}(t,x+h)+\varrho_{\alpha}^{0}(t,x)\right)+c_{4}\mathbf{1}_{\left\{ t<|h|\le1\right\} }(h)\varrho_{0}^{0}(t,x)|h|.\label{Lemma 9, eq1}
\end{align}
Note that
\begin{equation}
\int_{\Rd}\mathbf{1}_{\left\{ t<|h|\le1\right\} }(h)|h|\cdot\frac{1}{|h|^{d+1}}\mathrm{d}h=\int_{\left\{ t<|h|\le1\right\} }\frac{1}{|h|^{d}}\mathrm{d}h=c_{5}\ln(t^{-1}).\label{Lemma 9, eq2}
\end{equation}
Combining (\ref{Lemma 8, eq 3-1}), (\ref{Lemma 9, eq1}) and (\ref{Lemma 9, eq2}),
we obtain (\ref{esti:f_t,integral_differ, a=00003D1}).
\end{proof}
For a function $f$ on $\Rd$ we define the function $\delta_{f}$
on $\mathbb{R}^{2d}$ by
\[
\delta_{f}(x,x'):=f(x)-f(x'),\quad x,x'\in\Rd.
\]

\begin{lem}
Assume $\alpha\neq1$. Then there exists a constant $C_{22}=C_{22}(d,\alpha,\kappa_{0},\kappa_{1})>0$
such that for all $0<t\le1$ and $x,x'\in\Rd$,
\begin{align}
\int_{\Rd}\left|\delta_{f_{t}}(x+h,x'+h)-\delta_{f_{t}}(x,x')-\chi_{\alpha}(h)h\cdot\delta_{\nabla f_{t}}(x,x')\right|\cdot|h|^{-d-\alpha}\mathrm{d}h\qquad\nonumber \\
\le C_{22}\left(\left(t^{-1/\alpha}|x-x'|\right)\wedge1\right)\left\{ \varrho_{0}^{0}(t,x)+\varrho_{0}^{0}(t,x')\right\} .\label{esti:f_t,integral_two_differ}
\end{align}
\end{lem}
\begin{proof}
As in Lemma \ref{lem:frac esti f_t}, we only need to prove that for
all $0<t\le1$ and $x,x'\in\Rd$,
\begin{align*}
\int_{\Rd}\left|\delta_{g_{t}}(x+h,x'+h)-\delta_{g_{t}}(x,x')-\chi_{\alpha}(h)h\cdot\delta_{\nabla g_{t}}(x,x')\right|\cdot|h|^{-d-\alpha}\mathrm{d}h\qquad\\
\le\tilde{C}_{22}\left(\left(t^{-1/\alpha}|x-x'|\right)\wedge1\right)\left\{ \varrho_{0}^{0}(t,x)+\varrho_{0}^{0}(t,x')\right\} ,
\end{align*}
where $\tilde{C}_{22}=\tilde{C}_{22}(d,\alpha,\kappa_{0},\kappa_{1})>0$
is a constant.

(i) We first consider the case $\alpha>1$. If $|h|\le1$ and $|x-x'|\le1$,
then
\begin{align}
 & \left|\delta_{g_{1}}(x+h,x'+h)-\delta_{g_{1}}(x,x')-\chi_{\alpha}(h)h\cdot\delta_{\nabla g_{1}}(x,x')\right|\nonumber \\
 & \quad\overset{(\ref{identity 2: frac esti f_t})}{=}\left|\int_{0}^{1}\left(\int_{0}^{1}\left(\nabla^{2}g_{1}(x+rr'h)-\nabla^{2}g_{1}(x'+rr'h)\right)\cdot rh\mathrm{d}r'\right)\cdot h\mathrm{d}r\right|\nonumber \\
 & \quad\le c_{1}|h|^{2}|x-x'|\int_{0}^{1}\int_{0}^{1}\int_{0}^{1}\left|\nabla^{3}g_{1}(x+rr'h+r''(x'-x)\right|\mathrm{d}r''\mathrm{d}r'\mathrm{d}r\nonumber \\
 & \quad\overset{(\ref{bound:gradientg_t})}{\le}c_{2}|h|^{2}|x-x'|\int_{0}^{1}\int_{0}^{1}\int_{0}^{1}\left(1+|x+rr'h+r''(x'-x)|\right)^{-d-\alpha-3}\mathrm{d}r''\mathrm{d}r'\mathrm{d}r\nonumber \\
 & \quad\overset{(\ref{ineq 1: chen})}{\le}c_{3}|h|^{2}|x-x'|\left(1+|x|\right)^{-d-\alpha-3}\le c_{3}|h|^{2}|x-x'|\varrho_{\alpha}^{0}(1,x).\label{Lemma 10, eq 1}
\end{align}
If $|h|>1$ and $|x-x'|\le1$, then
\begin{align}
 & \left|\delta_{g_{1}}(x+h,x'+h)-\delta_{g_{1}}(x,x')-\chi_{\alpha}(h)h\cdot\delta_{\nabla g_{1}}(x,x')\right|\nonumber \\
 & \quad\le|x-x'|\int_{0}^{1}|\nabla g_{1}(x+h+r(x'-x))|\mathrm{d}r\nonumber \\
 & \qquad+|x-x'|\int_{0}^{1}|\nabla g_{1}(x+r(x'-x))|\mathrm{d}r\nonumber \\
 & \qquad+|h|\cdot|x-x'|\int_{0}^{1}|\nabla^{2}g_{1}(x+r(x'-x))|\mathrm{d}r\nonumber \\
 & \quad\overset{(\ref{bound:gradientg_t}),(\ref{ineq 1: chen})}{\le}c_{4}|x-x'|\left(1+|x+h|\right)^{-d-\alpha-1}+c_{4}|x-x'|\left(1+|x|\right)^{-d-\alpha-1}\nonumber \\
 & \qquad\qquad+c_{4}|h|\cdot|x-x'|\left(1+|x|\right)^{-d-\alpha-2}.\label{Lemma 10, eq2}
\end{align}
In view of (\ref{Lemma 8, eq1}), we thus get
\begin{align*}
 & \left|\delta_{g_{1}}(x+h,x'+h)-\delta_{g_{1}}(x,x')-\chi_{\alpha}(h)h\cdot\delta_{\nabla g_{1}}(x,x')\right|\\
 & \quad\le c_{5}\left(|x-x'|\wedge1\right)\left(|h|^{2}\wedge1\right)\left(\varrho_{\alpha}^{0}(1,x+h)+\varrho_{\alpha}^{0}(1,x)+\varrho_{\alpha}^{0}(1,x'+h)+\varrho_{\alpha}^{0}(1,x')\right)\\
 & \qquad+c_{6}\mathbf{1}_{\{|h|>1\}}|h|\left(|x-x'|\wedge1\right)\left(\varrho_{\alpha-1}^{0}(1,x)+\varrho_{\alpha-1}^{0}(1,x')\right).
\end{align*}
Then we can proceed in the same way as in the proof of Lemma \ref{lem:frac esti f_t}
to obtain (\ref{esti:f_t,integral_two_differ}).

(ii) Let $0<\alpha<1$. Similarly to (\ref{Lemma 10, eq 1}), we have
that for $|h|\le1$ and $|x-x'|\le1$,
\[
\left|\delta_{g_{1}}(x+h,x'+h)-\delta_{g_{1}}(x,x')\right|\le c_{7}|h|\cdot|x-x'|\varrho_{\alpha}^{0}(1,x);
\]
Similarly to (\ref{Lemma 10, eq2}), for $|h|>1$ and $|x-x'|\le1$,
we obtain
\begin{equation}
\left|\delta_{g_{1}}(x+h,x'+h)-\delta_{g_{1}}(x,x')\right|\le c_{8}|x-x'|\left(\varrho_{\alpha}^{0}(1,x)+\varrho_{\alpha}^{0}(1,x+h)\right).\label{new eq1: Lemma 11}
\end{equation}
Noting (\ref{new eq1: Lemma 9}), we thus get
\begin{align*}
 & \left|\delta_{g_{1}}(x+h,x'+h)-\delta_{g_{1}}(x,x')\right|\\
 & \quad\le c_{9}\left(|x-x'|\wedge1\right)\left(|h|\wedge1\right)\left(\varrho_{\alpha}^{0}(1,x+h)+\varrho_{\alpha}^{0}(1,x)+\varrho_{\alpha}^{0}(1,x'+h)+\varrho_{\alpha}^{0}(1,x')\right).
\end{align*}
The rest of the proof is completely similar to Lemma \ref{lem:frac esti f_t}.
We omit the details.
\end{proof}
\begin{lem}
\label{Lemma 2: delta}Assume $\alpha=1$. Then there exists a constant
$C_{23}=C_{23}(d,\alpha,\kappa_{0},\kappa_{1})>0$ such that for all
$0<t\le1$ and $x,x'\in\Rd$,
\begin{align}
\int_{\Rd}\lvert\delta_{f_{t}}(x+h,x'+h)-\delta_{f_{t}}(x,x')-\chi_{\alpha}(h)h\cdot\delta_{\nabla f_{t}}(x,x')\rvert\cdot\frac{1}{|h|^{d+\alpha}}\mathrm{d}h\qquad\nonumber \\
\le C_{23}\left(1+\ln(t^{-1})\right)\left(\left(t^{-1/\alpha}|x-x'|\right)\wedge1\right)\left\{ \varrho_{0}^{0}(t,x)+\varrho_{0}^{0}(t,x')\right\} .\label{esti:f_t,integral_two_differ-1}
\end{align}

\end{lem}
\begin{proof}
By (\ref{Lemma 10, eq 1}), (\ref{new eq1: Lemma 11}) and (\ref{eq1.5: frac esti f_t}),
we have
\begin{align*}
 & \lvert\delta_{g_{1}}(x+h,x'+h)-\delta_{g_{1}}(x,x')-\chi_{\alpha}(h)h\cdot\delta_{\nabla g_{1}}(x,x')\rvert\\
 & \le c_{1}\left(\big(|x-x'|\big)\wedge1\right)\left(|h|^{2}\wedge1\right)\Big(\varrho_{\alpha}^{0}(1,x+h)+\varrho_{\alpha}^{0}(1,x)+\varrho_{\alpha}^{0}(1,x'+h)+\varrho_{\alpha}^{0}(1,x')\Big).
\end{align*}
 Similarly to (\ref{Lemma 10, eq 1}), if $t^{-1}|x-x'|\le1$, then
\[
|\nabla g_{1}(t^{-1}x)-\nabla g_{1}(t^{-1}x')|\le c_{2}t^{-1}|x-x'|(1+|t^{-1}x|)^{-d-3}.
\]
Noting (\ref{bound:gradientg_t}), we thus get
\begin{align*}
 & |\nabla g_{1}(t^{-1}x)-\nabla g_{1}(t^{-1}x')|\\
 & \quad\le c_{3}\left((t^{-1}|x-x'|)\wedge1\right)\Big((1+|t^{-1}x|)^{-d-2}+(1+|t^{-1}x'|)^{-d-2}\Big).
\end{align*}
The rest of the proof goes in the same way as in Lemma \ref{lem:Assume ln}.
\end{proof}

\section{Transition density of the Markov process associated with $\mathcal{A}$ }

In this section we will use Levi's method (parametrix) to construct
the transition density of the Markov processes that corresponds to
the generator $\mathcal{A}$, where
\begin{equation}
\mathcal{A}f(x)=\int_{\mathbb{R}^{d}\backslash\{0\}}\left[f(x+h)-f(x)-\chi_{\alpha}(h)h\cdot\nabla f(x)\right]\frac{n(x,h)}{|h|^{d+\alpha}}\mathrm{d}h.\label{defi 2: A}
\end{equation}
Throughout this section, we assume that $n(\cdot,\cdot)$ satisfies
Assumption \ref{ass1:the function n}.

Levi's method has been applied in \cite{chen2015heat} and \cite{MR1744782}
to construct transition densities of stable-like processes that are
similar to what we consider here. In the sequel we will follow closely
the approach of \cite{chen2015heat}.

According to Assumption \ref{ass1:the function n}, for each $y\in\Rd$,
$h\mapsto n(y,h)|h|^{-d-\alpha}$ is a function that satisfies (\ref{condition1forK})
and (\ref{condition2forK_a=00003D00003D1}). Let $f_{t}^{y}(\cdot)$,
$t>0$, be the density functions of the stable-like Lévy process with
the jump kernel $n(y,h)|h|^{-d-\alpha}$, namely,
\begin{equation}
f_{t}^{y}(x):=\frac{1}{(2\pi)^{d}}\int_{\mathbb{R}^{d}}e^{-iu\cdot x}e^{-t\psi^{y}(u)}\mathrm{d}u,\quad x\in\mathbb{R}^{d},\ t>0,\label{defi: f_t^y}
\end{equation}
where
\begin{equation}
\psi^{y}(u)=-\int_{\mathbb{R}^{d}\setminus\{0\}}\Big(e^{iu\cdot h}-1-\chi_{\alpha}(h)iu\cdot h\Big)\frac{n(y,h)}{|h|^{d+\alpha}}\mathrm{d}h.\label{defi: psi^y}
\end{equation}
 Define the operator $\mathcal{A}^{y}$ by
\begin{equation}
\mathcal{A}^{y}f(x):=\int_{\mathbb{R}^{d}\backslash\{0\}}\left[f(x+h)-f(x)-\chi_{\alpha}(h)h\cdot\nabla f(x)\right]\frac{n(y,h)}{|h|^{d+\alpha}}\mathrm{d}h.\label{defi: A^y}
\end{equation}

\begin{rem}
In view of Assumption \ref{ass1:the function n}, all the estimates
that we established in Lemmas \ref{lem: upper esti for f_t} \textendash{}
\ref{Lemma 2: delta} are also true for $f_{t}^{y}$ (in place of
$f_{t}$). \label{rem: f_t^y}

The following Lemma is analog to \cite[Theorem 2.5]{chen2015heat}.
\end{rem}
\begin{lem}
\label{Lemma 12}Suppose $\gamma\in(0,\alpha/4)$. Then there exists
some constant $C_{24}=C_{24}(d,\alpha,\kappa_{0},\kappa_{1},\kappa_{2},\gamma)>0$
such that for all $0<t\le1$ and $x,x'\in\Rd$,
\begin{equation}
\left|f_{t}^{y}(x)-f_{t}^{y'}(x)\right|\le C_{24}\left(|y-y'|^{\theta}\wedge1\right)\left(\varrho_{\alpha}^{0}+\varrho_{\alpha-\gamma}^{\gamma}\right)(t,x),\label{esti:f_t_differof_y_y'}
\end{equation}
\begin{equation}
\left|\nabla_{x}f_{t}^{y}(x)-\nabla_{x}f_{t}^{y'}(x)\right|\le C_{24}\left(|y-y'|^{\theta}\wedge1\right)\left(\varrho_{\alpha-1}^{0}+\varrho_{\alpha-\gamma-1}^{\gamma}\right)(t,x),\label{esti:gradient_f_t_differof_y_y'}
\end{equation}
and
\begin{align}
\int_{\Rd}\left|\left(f_{t}^{y}-f_{t}^{y'}\right)(x+h)-\left(f_{t}^{y}-f_{t}^{y'}\right)(x)-\chi_{\alpha}(h)h\cdot\nabla\left(f_{t}^{y}-f_{t}^{y'}\right)(x)\right|\cdot|h|^{-d-\alpha}\mathrm{d}h\nonumber \\
\le C_{24}\left(|y-y'|^{\theta}\wedge1\right)\left(\varrho_{0}^{0}+\varrho_{-\gamma}^{\gamma}\right)(t,x).\label{esti:f_t_integral_differ_of_y_y'}
\end{align}

\end{lem}
\begin{proof}
The proof is almost the same as that of \cite[Theorem 2.5]{chen2015heat},
and we only need to verify that for $t>0,\,x,y,y'\in\Rd$,
\begin{equation}
f_{t}^{y}(x)-f_{t}^{y'}(x)=\int_{0}^{t}\int_{\mathbb{R}^{d}}\left(f_{t-s}^{y'}(z)-f_{t-s}^{y'}(x)\right)(\mathcal{A}^{y}-\mathcal{A}^{y'})\left(f_{s}^{y}(x-\cdot)\right)(z)\mathrm{d}z\mathrm{d}s.\label{identity: Duhamel f_t}
\end{equation}
By (\ref{defi: f_t^y}) and (\ref{defi: A^y}), we have
\[
(\mathcal{A}^{y}-\mathcal{A}^{y'})\left(f_{s}^{y}(x-\cdot)\right)(z)=-\frac{1}{(2\pi)^{d}}\int_{\Rd}\left(\psi^{y}(u)-\psi^{y'}(u)\right)e^{-s\psi^{y}(u)}e^{-iu\cdot(x-z)}\mathrm{d}u.
\]
Note that $\int_{\mathbb{R}^{d}}(\mathcal{A}^{y}-\mathcal{A}^{y'})\left(f_{s}^{y}(x-\cdot)\right)(z)\mathrm{d}z=0.$
By the Fubini's theorem, we have that for $0<\varepsilon<t$,
\begin{align}
 & \int_{\varepsilon}^{t}\int_{\mathbb{R}^{d}}\left(f_{t-s}^{y'}(z)-f_{t-s}^{y'}(x)\right)(\mathcal{A}^{y}-\mathcal{A}^{y'})\left(f_{s}^{y}(x-\cdot)\right)(z)\mathrm{d}z\mathrm{d}s\nonumber \\
 & \quad=\int_{\varepsilon}^{t}\int_{\mathbb{R}^{d}}f_{t-s}^{y'}(z)(\mathcal{A}^{y}-\mathcal{A}^{y'})\left(f_{s}^{y}(x-\cdot)\right)(z)\mathrm{d}z\mathrm{d}s\nonumber \\
 & \quad=-\frac{1}{(2\pi)^{d}}\int_{\varepsilon}^{t}\int_{\mathbb{R}^{d}}f_{t-s}^{y'}(z)\left(\int_{\Rd}\left(\psi^{y}(u)-\psi^{y'}(u)\right)e^{-s\psi^{y}(u)}e^{-iu\cdot(x-z)}\mathrm{d}u\right)\mathrm{d}z\mathrm{d}s\nonumber \\
 & \quad=-\frac{1}{(2\pi)^{d}}\int_{\varepsilon}^{t}\int_{\mathbb{R}^{d}}\left(\psi^{y}(u)-\psi^{y'}(u)\right)e^{-s\psi^{y}(u)-iu\cdot x}e^{-(t-s)\psi^{y'}(u)}\mathrm{d}u\mathrm{d}s\nonumber \\
 & \quad=\frac{1}{(2\pi)^{d}}\int_{\mathbb{R}^{d}}e^{-iu\cdot x-t\psi^{y'}(u)}\left(e^{-t\psi^{y}(u)}e^{t\psi^{y'}(u)}-e^{-\varepsilon\psi^{y}(u)}e^{\varepsilon\psi^{y'}(u)}\right)\mathrm{d}u\nonumber \\
 & \quad=f_{t}^{y}(x)-\frac{1}{(2\pi)^{d}}\int_{\mathbb{R}^{d}}e^{-iu\cdot x-(t-\varepsilon)\psi^{y'}(u)}e^{-\varepsilon\psi^{y}(u)}\mathrm{d}u.\label{Prop 1, eq 1}
\end{align}

By (\ref{esti2:rho}), (\ref{MPAFLsect31-1}), (\ref{Coro 1, eq})
and the dominated convergence theorem, we can let $\varepsilon\to0$
in (\ref{Prop 1, eq 1}) to obtain (\ref{identity: Duhamel f_t}).
\end{proof}
For $t\in(0,1]$ and $x,y\in\Rd$, define
\begin{equation}
q(t,x,y):=f_{t}^{y}(y-x)\label{defi: q(t,x,y)}
\end{equation}
and
\begin{align*}
F(t,x,y):= & \left(\mathcal{A}-\mathcal{A}^{y}\right)q(t,\cdot,y)(x)\\
= & \int_{\mathbb{R}^{d}\backslash\{0\}}\Big[q(t,x+h,y)-q(t,x,y)\\
 & \qquad-\chi_{\alpha}(h)h\cdot\nabla_{x}q(t,x,y)\Big]\frac{\left(n(x,h)-n(y,h)\right)}{|h|^{d+\alpha}}\mathrm{d}h.
\end{align*}
For functions $\varphi_{1},\varphi_{2}$ on $(0,1]\times\Rd\times\Rd$,
we introduce the notation $\varphi_{1}\otimes\varphi_{2}$ by
\[
\varphi_{1}\otimes\varphi_{2}(t,x,y):=\int_{0}^{t}\int_{\Rd}\varphi_{1}(t-s,x,z)\varphi_{2}(s,z,y)\mathrm{d}z\mathrm{d}s,\quad t\in(0,1],\ x,y\in\Rd.
\]

Next, we study the convergence of the series $\sum_{n=1}^{\infty}F^{\otimes n}$,
where $F^{\otimes1}:=F$ and $F^{\otimes n}:=F\otimes\left(F^{\otimes(n-1)}\right)$.
Recall that the constant $\theta$ is given in (\ref{Condition:Holder}).
In the rest of this paper, let $\hat{\theta}:=\theta\wedge(\alpha/4).$
\begin{lem}
 \noun{(}\emph{i}\noun{) }Define
\begin{equation}
\Phi(t,x,y):=\sum_{n=1}^{\infty}F^{\otimes n}(t,x,y),\quad(t,x,y)\in(0,1]\times\Rd\times\Rd.\label{defi Phi}
\end{equation}
Then the series on the right-hand side of \emph{(\ref{defi Phi})
converges locally uniformly on $(0,1]\times\Rd\times\Rd$. Moreover,}
$\Phi$ is continuous on $(0,1]\times\Rd\times\Rd$, and there exists
a constant $C_{26}=C_{26}(d,\alpha,\kappa_{0},\kappa_{1},\kappa_{2},\theta)>0$
such that for all $(t,x,y)\in(0,1]\times\Rd\times\Rd$,
\begin{equation}
|\Phi(t,x,y)|\le C_{26}\left(\varrho_{\hat{\theta}}^{0}(t,x-y)+\varrho_{0}^{\hat{\theta}}(t,x-y)\right).\label{esti:Phi}
\end{equation}
\noun{(}\emph{ii}\noun{)} Given $\gamma\in(0,\hat{\theta})$, there
exists a constant $C_{27}=C_{27}(d,\alpha,\kappa_{0},\kappa_{1},\kappa_{2},\theta,\gamma)>0$
such that for all $t\in(0,1]$ and $x,x',y\in\Rd$,
\begin{align*}
 & |\Phi(t,x,y)-\Phi(t,x',y)|\\
 & \qquad\le C_{27}\left(|x-x'|^{\hat{\theta}-\gamma}\wedge1\right)\left\{ \left(\varrho_{\gamma}^{0}+\varrho_{\gamma-\hat{\theta}}^{\hat{\theta}}\right)(t,x-y)+\left(\varrho_{\gamma}^{0}+\varrho_{\gamma-\hat{\theta}}^{\hat{\theta}}\right)(t,x'-y)\right\} .
\end{align*}
 \label{lem:esti_F^n}
\end{lem}
\begin{proof}
In view of Lemma \ref{lem: upper esti for f_t} \textendash{} Lemma
\ref{Lemma 12} and Remark \ref{rem: f_t^y}, the proof is essentially
the same as in \cite[Theorem 4.1]{chen2015heat}. We omit the details.
\end{proof}
By (\ref{esti3:rho}), (\ref{upperbound:f_t}) and (\ref{esti:Phi}),
there exists a constant $C_{28}=C_{28}(d,\alpha,\kappa_{0},\kappa_{1},\kappa_{2},\theta)>0$
such that
\begin{equation}
q\otimes\Phi(t,x,y)\le C_{28}\left(\varrho_{\alpha+\hat{\theta}}^{0}+\varrho_{\alpha}^{\hat{\theta}}\right)(t,x-y),\quad t\in(0,1],\ x,y\in\Rd.\label{esti: q tensor Phi}
\end{equation}
It follows that
\begin{equation}
p(t,x,y):=q(t,x,y)+q\otimes\Phi(t,x,y),\quad(t,x,y)\in(0,1]\times\Rd\times\Rd,\label{defi: p(t,x,y)}
\end{equation}
is well-defined.
\begin{prop}
\label{Prop. esti: p(t,x,y) }There exists a constant $C_{29}=C_{29}(d,\alpha,\kappa_{0},\kappa_{1},\kappa_{2},\theta)>0$
such that
\begin{align}
|p(t,x,y)| & \le C_{29}\varrho_{\alpha}^{0}(t,x-y),\quad(t,x,y)\in(0,1]\times\Rd\times\Rd.\label{main esti: p}
\end{align}
Moreover, the function $(t,x,y)\mapsto p(t,x,y)$ is continuous on
$(0,1]\times\Rd\times\Rd$.
\end{prop}
\begin{proof}
The estimate (\ref{main esti: p}) is a simple consequence of (\ref{upperbound:f_t})
and (\ref{esti: q tensor Phi}). By (\ref{MPAFLsect31-1}) and Assumption
\ref{ass1:the function n}, there exists a constant $c_{1}=c_{1}(d,\alpha,\kappa_{0})>0$
with
\[
|\exp(-iu\cdot x-t\psi^{y}(u))|\le\exp(-c_{1}t|u|^{\alpha}),\quad\forall t>0,\ x,y,u\in\Rd,
\]
where $\psi^{y}$ is given in (\ref{defi: psi^y}). The continuity
of $(t,x,y)\mapsto q(t,x,y)$ now follows from (\ref{defi: q(t,x,y)}),
(\ref{defi: f_t^y}) and the dominated convergence. Since $q(t,x,y)$
and $\Phi(t,x,y)$ are both continuous, again by dominated convergence,
the function $(0,1]\times\Rd\times\Rd\ni(t,x,y)\mapsto p(t,x,y)$
is also continuous.
\end{proof}
In the remaining part of this section we will show that $p(t,x,y)$
is the transition density of the Markov process associated with $\mathcal{A}$.
The ideas for the proof of the next two propositions come from \cite[Chap.~ 1, Theorems 4 - 5]{MR0181836}.
\begin{prop}
Suppose that the function $\varphi:(0,1]\times\Rd\to\mathbb{R}$ is
continuous and such that for all $x,x'\in\Rd$ and $t\in(0,1]$,
\begin{equation}
|\varphi(t,x)|\le c_{\varphi}t^{-1+\hat{\theta}/\alpha}\label{condition1:f(t,x)}
\end{equation}
and
\begin{equation}
|\varphi(t,x)-\varphi(t,x')|\le c_{\varphi}t^{-1+\gamma/\alpha}\left(|x-x'|^{\hat{\theta}-\gamma}\wedge1\right),\label{condition2:f(t,x)}
\end{equation}
 where $c_{\varphi}>0$ and $\gamma\in(0,\hat{\theta})$ are constants.
Consider the function $V$ defined by
\begin{equation}
V(t,x):=\int_{0}^{t}\int_{\Rd}q(t-s,x,z)\varphi(s,z)\mathrm{d}z\mathrm{d}s,\quad(t,x)\in(0,1]\times\Rd.\label{defi:V}
\end{equation}
Then for each $t\in(0,1]$, $\mathcal{A}V(t,\cdot)$ is well-defined
and
\begin{equation}
\mathcal{A}V(t,\cdot)(x)=\int_{0}^{t}\int_{\Rd}\mathcal{A}q(t-s,\cdot,z)(x)\varphi(s,z)\mathrm{d}z\mathrm{d}s,\quad x\in\Rd.\label{eq:AV}
\end{equation}
We also have the estimate
\begin{equation}
{\color{red}{\color{black}\left|\mathcal{A}V(t,\cdot)(x)\right|\le C_{30}\left(1+\ln\left(t^{-1}\right)\right)t^{-1},\quad\forall x\in\Rd,\ t\in(0,1],}}\label{esti: prop 2, AV}
\end{equation}
where $C_{30}=C_{30}(d,\alpha,\kappa_{0},\kappa_{1},\kappa_{2},\theta,\gamma,c_{\varphi})>0$
is a constant. \label{prop1:layerpotential}
\end{prop}
\begin{proof}
Let $0<s<t\le1$ and $x\in\Rd$ be arbitrary. By (\ref{upperbound:f_t})
and (\ref{condition1:f(t,x)}), we have
\begin{equation}
\int_{\Rd}\lvert q(t-s,x,z)\varphi(s,z)\rvert\mathrm{d}z\le c_{1}\int_{\Rd}\varrho_{\alpha}^{0}(t-s,x-z)s^{-1+\hat{\theta}/\alpha}\mathrm{d}z\overset{(\ref{esti1:rho})}{\le}c_{2}s^{-1+\hat{\theta}/\alpha}.\label{eq 1: conti. of V}
\end{equation}
So the function $V$ in (\ref{defi:V}) is well-defined. Let
\begin{equation}
J(t,s,x):=\int_{\Rd}q(t-s,x,z)\varphi(s,z)\mathrm{d}z.\label{defi:J(t,s,x)}
\end{equation}
By (\ref{upperbound:gradientf_t}), (\ref{esti:f_t,integral_differ})
and (\ref{esti:f_t,integral_differ, a=00003D1}), we obtain that for
$|x-x_{0}|\le(t-s)^{1/\alpha}$,
\begin{align*}
\lvert\nabla_{x}q(t-s,x,z)\rvert & \le c_{3}\varrho_{\alpha-1}^{0}(t-s,x-z)\overset{(\ref{ineq 1: chen})}{\le}c_{4}\varrho_{\alpha-1}^{0}(t-s,x_{0}-z).
\end{align*}
So it is easy to see that for $0<s<t\le1$ and $x\in\Rd$,
\begin{equation}
\nabla_{x}J(t,s,x)=\int_{\Rd}\nabla_{x}q(t-s,x,z)\varphi(s,z)\mathrm{d}z.\label{eq: differentiability of J in x}
\end{equation}
Similarly, we have
\begin{equation}
\left|\mathcal{A}q(t-s,\cdot,z)(x)\right|\le c_{5}\left(1+\ln\left(\left(t-s\right)^{-1}\right)\right)\varrho_{0}^{0}(t-s,x-z)\label{esti: Aq}
\end{equation}
and
\begin{equation}
\mathcal{A}J(t,s,\cdot)(x)=\int_{\Rd}\mathcal{A}q(t-s,\cdot,z)(x)\varphi(s,z)\mathrm{d}z.\label{eq: AJ=00003Dint Aq ds}
\end{equation}
Let $y\in\Rd$ be arbitrary. We now write
\begin{align}
J(t,s,x) & =\int_{\Rd}q(t-s,x,z)\left(\varphi(s,z)-\varphi(s,y)\right)\mathrm{d}z\nonumber \\
 & \qquad+\varphi(s,y)\int_{\Rd}\left(q(t-s,x,z)-f_{t-s}^{y}(z-x)\right)\mathrm{d}z+\varphi(s,y).\label{eq:J(t,s,x) with y}
\end{align}

We will complete the proof in two steps.

``$Step$ 1'': We show that if $\alpha\ge1$, then
\begin{equation}
\nabla_{x}V(t,x)=\int_{0}^{t}\nabla_{x}J(t,s,x)ds,\quad(t,x)\in(0,1]\times\Rd.\label{eq: grad V}
\end{equation}
By (\ref{defi: q(t,x,y)}), (\ref{condition1:f(t,x)}), (\ref{condition2:f(t,x)})
and (\ref{eq:J(t,s,x) with y}), we have
\begin{align}
\left|\nabla_{x}J(t,s,x)\right| & \le\left|\int_{\Rd}\nabla_{x}\left(f_{t-s}^{z}(z-x)\right)\left(\varphi(s,z)-\varphi(s,y)\right)\mathrm{d}z\right|\nonumber \\
 & \quad+|\varphi(s,y)|\cdot\left|\int_{\Rd}\left(\nabla_{x}\left(f_{t-s}^{z}(z-x)\right)-\nabla_{x}\left(f_{t-s}^{y}(z-x)\right)\right)\hbox{d}z\right|\nonumber \\
 & \overset{(\ref{upperbound:gradientf_t}),(\ref{esti:gradient_f_t_differof_y_y'})}{{\le}}c_{6}\int_{\Rd}s^{-1+\gamma/\alpha}\left(|y-z|^{\hat{\theta}-\gamma}\wedge1\right)\varrho_{\alpha-1}^{0}(t-s,x-z)\hbox{d}z\nonumber \\
 & \quad+c_{7}s^{-1+\hat{\theta}/\alpha}\int_{\Rd}\left(|y-z|^{\hat{\theta}}\wedge1\right)\left(\varrho_{\alpha-1}^{0}+\varrho_{\alpha-\gamma-1}^{\gamma}\right)(t-s,x-z)\hbox{d}z,\label{esti1:grad J for a>1}
\end{align}
where the constants $c_{5}$ and $c_{6}$ are independent of $y$.
Choosing $y=x$ in (\ref{esti1:grad J for a>1}), we get
\begin{align}
\left|\nabla_{x}J(t,s,x)\right| & \overset{(\ref{esti1:rho})}{\le}c_{8}s^{-1+\gamma/\alpha}(t-s)^{(\hat{\theta}-\gamma-1)/\alpha}+c_{9}s^{-1+\hat{\theta}/\alpha}(t-s)^{(\hat{\theta}-1)/\alpha}.\label{esti2:grad J for a>1}
\end{align}
If $\alpha\ge1$, then the right-hand side of (\ref{esti2:grad J for a>1}),
as a function with the variable $s$, is integrable on $[0,t]$. The
equation (\ref{eq: grad V}) now follows by the dominated convergence
theorem.

``$Step$ 2'': We consider a general $\alpha\in(0,2)$ and show
that $\mathcal{A}V(t,\cdot)(x)$ is well-defined and (\ref{eq:AV})
holds. For $h\in\Rd$ and $h\neq0$, it follows from (\ref{eq: grad V})
that
\begin{align}
 & V(t,x+h)-V(t,x)-\chi_{\alpha}(h)h\cdot\nabla_{x}V(t,x)\nonumber \\
 & \quad=\int_{0}^{t}\left[J(t,s,x+h)-J(t,s,x)-\chi_{\alpha}(h)h\cdot\nabla_{x}J(t,s,x)\right]\mathrm{d}s.\label{esti1: delta}
\end{align}
 By (\ref{defi: q(t,x,y)}), (\ref{condition1:f(t,x)}), (\ref{condition2:f(t,x)})
and (\ref{eq:J(t,s,x) with y}), we get
\begin{align}
 & \left|J(t,s,x+h)-J(t,s,x)-\chi_{\alpha}(h)h\cdot\nabla_{x}J(t,s,x)\right|\nonumber \\
 & \quad\le c_{10}\int_{\Rd}|f_{t-s}^{z}(z-x-h)-f_{t-s}^{z}(z-x)-\chi_{\alpha}(h)h\cdot\nabla_{x}\left(f_{t-s}^{z}(z-x)\right)|\,s^{-1+\gamma/\alpha}\nonumber \\
 & \qquad\times\left(|y-z|^{\hat{\theta}-\gamma}\wedge1\right)\hbox{d}z+c_{11}s^{-1+\hat{\theta}/\alpha}\int_{\Rd}|f_{t-s}^{z}(z-x-h)-f_{t-s}^{y}(z-x-h)\nonumber \\
 & \qquad\quad-f_{t-s}^{z}(z-x)+f_{t-s}^{y}(z-x)-\chi_{\alpha}(h)h\cdot\nabla_{x}\left(f_{t-s}^{z}(z-x)\right)\nonumber \\
 & \qquad\qquad+\chi_{\alpha}(h)h\cdot\nabla_{x}\left(f_{t-s}^{y}(z-x)\right)|\mathrm{d}z.\label{eq: prop 2, I}
\end{align}
It follows from (\ref{esti:f_t,integral_differ}), (\ref{esti:f_t,integral_differ, a=00003D1}),
(\ref{esti:f_t_integral_differ_of_y_y'}), (\ref{eq: prop 2, I})
and the Fubini's theorem that
\begin{align}
 & I(t,s,x):=\int_{\Rd\setminus\left\{ 0\right\} }\left|J(t,s,x+h)-J(t,s,x)-\chi_{\alpha}(h)h\cdot\nabla_{x}J(t,s,x)\right|\cdot\frac{n(x,h)}{|h|^{d+\alpha}}\hbox{d}h\nonumber \\
 & \le c_{12}\left(1+\ln\left[(t-s)^{-1}\right]\right)\int_{\Rd}\varrho_{0}^{0}(t-s,x-z)s^{-1+\gamma/\alpha}\left(|y-z|^{\hat{\theta}-\gamma}\wedge1\right)\hbox{d}z\nonumber \\
 & \ +c_{13}s^{-1+\hat{\theta}/\alpha}\left(1+\ln\left[(t-s)^{-1}\right]\right)\int_{\Rd}\left(|y-z|^{\hat{\theta}}\wedge1\right)\left(\varrho_{0}^{0}+\varrho_{-\gamma}^{\gamma}\right)(t-s,x-z)\hbox{d}z.\label{esti1: AJ(t,s,x)}
\end{align}
Choosing $y=x$ in (\ref{esti1: AJ(t,s,x)}) and applying (\ref{esti1:rho}),
we get
\begin{align}
 & \quad I(t,s,x)\le c_{14}\left(1+\ln\left[(t-s)^{-1}\right]\right)s^{-1+\gamma/\alpha}(t-s)^{(\hat{\theta}-\gamma-\alpha)/\alpha}\nonumber \\
 & \qquad+c_{15}\left(1+\ln\left[(t-s)^{-1}\right]\right)s^{-1+\hat{\theta}/\alpha}(t-s)^{(\hat{\theta}-\alpha)/\alpha}\nonumber \\
 & \quad\overset{\gamma\in(0,\hat{\theta})}{\le}c_{16}\left(1+\ln\left[(t-s)^{-1}\right]\right)s^{-1+\gamma/\alpha}(t-s)^{(\hat{\theta}-\gamma-\alpha)/\alpha},\label{esti2: AJ(t,s,x)}
\end{align}
which implies
\begin{align}
 & \quad\int_{0}^{t}I(t,s,x)\mathrm{d}s\le c_{16}\int_{0}^{t}\left(1+\ln\left[(t-s)^{-1}\right]\right)s^{-1+\gamma/\alpha}(t-s)^{(\hat{\theta}-\gamma-\alpha)/\alpha}\mathrm{d}s\nonumber \\
 & \quad\le c_{16}\int_{0}^{t/2}\left(1+\ln\left[\left(\frac{t}{2}\right)^{-1}\right]\right)s^{-1+\gamma/\alpha}\left(\frac{t}{2}\right)^{(\hat{\theta}-\gamma-\alpha)/\alpha}\mathrm{d}s\nonumber \\
 & \qquad+c_{16}\int_{t/2}^{t}\left(1+\ln\left[(t-s)^{-1}\right]\right)\left(\frac{t}{2}\right)^{-1+\gamma/\alpha}(t-s)^{(\hat{\theta}-\gamma-\alpha)/\alpha}\mathrm{d}s\nonumber \\
 & \quad\le c_{17}\left(1+\ln\left(t^{-1}\right)\right)t^{(\hat{\theta}-\alpha)/\alpha}+c_{17}t^{-1+\gamma/\alpha}t^{(\hat{\theta}-\gamma)/\left(2\alpha\right)}\nonumber \\
 & \quad\le c_{18}\left(1+\ln\left(t^{-1}\right)\right)t^{-1}.\label{bound for AV on =00005Be,1=00005D}
\end{align}
 So $\mathcal{A}V(t,\cdot)(x)$ is well-defined and (\ref{esti: prop 2, AV})
is true. By (\ref{eq: AJ=00003Dint Aq ds}), (\ref{esti1: delta})
and the Fubini's theorem, we obtain
\[
\mathcal{A}V(t,\cdot)(x)=\int_{0}^{t}\mathcal{A}J(t,s,\cdot)(x)\mathrm{d}s=\int_{0}^{t}\int_{\Rd}\mathcal{A}q(t-s,\cdot,z)(x)\varphi(s,z)\mathrm{d}z\mathrm{d}s.
\]

This completes the proof.
\end{proof}
\begin{prop}
Let $\varphi$ and $V$ be as in Proposition \ref{prop1:layerpotential}.
Then for all $t\in(0,1]$ and $x\in\Rd$, $\partial_{t}V(t,x)$ exists
and satisfies
\begin{equation}
\partial_{t}V(t,x)=\varphi(t,x)+\int_{0}^{t}\int_{\Rd}\mathcal{A}^{z}q(t-s,\cdot,z)(x)\varphi(s,z)\mathrm{d}z\mathrm{d}s.\label{eq:partqF^n}
\end{equation}
Moreover, for each $x\in\Rd$, $t\mapsto\partial_{t}V(t,x)$ is continuous
on $(0,1]$.\label{prop2:layerpotential}
\end{prop}
\begin{proof}
Let $J$ be the same as in (\ref{defi:J(t,s,x)}). It is easy to verify
that $\partial_{t}J(t,s,x)$ exists for $0<s<t\le1$ and $x\in\Rd$.

Let $x\in\Rd$ be fixed. We only consider the case with $0<t<1$,
$h>0$ and $t+h\le1$, since the argument we will use works similarly
when $0<t-h<t\le1$. We have
\begin{align}
 & h^{-1}\left(V(t+h,x)-V(t,x)\right)\nonumber \\
 & \quad=h^{-1}\int_{0}^{t+h}J(t+h,s,x)\mathrm{d}s-h^{-1}\int_{0}^{t}J(t,s,x)\mathrm{d}s\nonumber \\
 & \quad=h^{-1}\int_{t}^{t+h}J(t+h,s,x)\mathrm{d}s+\int_{0}^{t}h^{-1}\left[J(t+h,s,x)-J(t,s,x)\right]\mathrm{d}s\nonumber \\
 & \quad=h^{-1}\int_{t}^{t+h}\left[J(t+h,s,x)-\varphi(t,x)\right]\mathrm{d}s+\varphi(t,x)+\int_{0}^{t}J_{1}(t^{*},s,x)\mathrm{d}s,\label{eq: prop 3, *}
\end{align}
where $J_{1}(t,s,x):=\partial_{t}J(t,s,x)$ and $t^{*}\in[t,t+h]$.

We will complete the proof in several steps.

``$Step$ 1'': We show that
\begin{equation}
\lim_{h\downarrow0}h^{-1}\int_{t}^{t+h}\left|J(t+h,s,x)-\varphi(t,x)\right|\mathrm{d}s=0.\label{esti0: J-phi}
\end{equation}
For $s\in(t,t+h)$, we have
\begin{align}
 & |J(t+h,s,x)-\varphi(t,x)|\nonumber \\
 & \quad=\Big\lvert\int_{\Rd}\left[q(t+h-s,x,z)-f_{t+h-s}^{x}(z-x)\right]\varphi(s,z)\mathrm{d}z\nonumber \\
 & \quad\qquad+\int_{\Rd}f_{t+h-s}^{x}(z-x)\left[\varphi(s,z)-\varphi(t,x)\right]\mathrm{d}z\Big\rvert\nonumber \\
 & \quad\overset{(\ref{defi: q(t,x,y)})}{\le}\int_{\Rd}\left|f_{t+h-s}^{z}(z-x)-f_{t+h-s}^{x}(z-x)\right|\cdot|\varphi(s,z)\mathrm{|d}z\nonumber \\
 & \quad\qquad+\int_{\Rd}f_{t+h-s}^{x}(z-x)\cdot|\varphi(s,z)-\varphi(t,x)|\mathrm{d}z\nonumber \\
 & \quad=:I_{1}+I_{2}.\label{esti1: J-phi}
\end{align}
For $I_{1}$, by (\ref{esti:f_t_differof_y_y'}), (\ref{condition1:f(t,x)})
and noting that $s\in(t,t+h)$, we have
\begin{align}
 & I_{1}\le c_{1}s^{-1+\hat{\theta}/\alpha}\int_{\Rd}\left(|z-x|^{\hat{\theta}}\wedge1\right)\left(\varrho_{\alpha}^{0}+\varrho_{\alpha-\gamma}^{\gamma}\right)(t+h-s,z-x)\mathrm{d}z\nonumber \\
 & \quad\overset{(\ref{esti1:rho})}{\le}c_{2}t^{-1+\hat{\theta}/\alpha}\left(t+h-s\right)^{\hat{\theta}/\alpha}\le c_{2}t^{-1+\hat{\theta}/\alpha}h^{\hat{\theta}/\alpha}.\label{esti11: J-phi}
\end{align}
For $I_{2}$ and $n\in\mathbb{N}$, by (\ref{upperbound:f_t}), (\ref{condition1:f(t,x)})
and noting that $s\in(t,t+h)$, we have
\begin{align}
 & I_{2}\le c_{3}\int_{\left\{ |z-x|\ge1/n\right\} }\varrho_{\alpha}^{0}(t+h-s,z-x)\cdot|\varphi(s,z)-\varphi(t,x)|\mathrm{d}z\nonumber \\
 & \quad\qquad+c_{3}\int_{\left\{ |z-x|\le1/n\right\} }\varrho_{\alpha}^{0}(t+h-s,z-x)\cdot|\varphi(s,z)-\varphi(t,x)|\mathrm{d}z\nonumber \\
 & \quad\le c_{4}t^{-1+\hat{\theta}/\alpha}\int_{\left\{ |z-x|\ge1/n\right\} }\varrho_{\alpha}^{0}(t+h-s,z-x)\mathrm{d}z\nonumber \\
 & \quad\qquad+c_{3}\int_{\left\{ |z-x|\le1/n\right\} }\varrho_{\alpha}^{0}(t+h-s,z-x)\cdot|\varphi(s,z)-\varphi(t,x)|\mathrm{d}z.\label{esti2: J-phi}
\end{align}
For any given $\varepsilon>0$, by the continuity of $\varphi$, we
can find $n_{0}\in\mathbb{N}$ and $h_{0}>0$ such that
\begin{equation}
|\varphi(s,z)-\varphi(t,x)|<\varepsilon,\quad\forall s\in(t,t+h_{0}),\,|z-x|\le\frac{1}{n_{0}}.\label{esti3: J-phi}
\end{equation}
By (\ref{esti2: J-phi}) and (\ref{esti3: J-phi}), we get that for
$t<s<t+h<t+h_{0}$,
\begin{align}
I_{2} & \le c_{4}t^{-1+\hat{\theta}/\alpha}\int_{\left\{ |z-x|\ge1/n_{0}\right\} }\varrho_{\alpha}^{0}(t+h-s,z-x)\mathrm{d}z+c_{5}\varepsilon\nonumber \\
 & =c_{4}t^{-1+\hat{\theta}/\alpha}\int_{\left\{ |z|\ge1/n_{0}\right\} }\varrho_{\alpha}^{0}(t+h-s,z)\mathrm{d}z+c_{5}\varepsilon\nonumber \\
 & =c_{4}t^{-1+\hat{\theta}/\alpha}\int_{\left\{ |z'|\ge\left(t+h-s\right)^{-1/\alpha}/n_{0}\right\} }\varrho_{\alpha}^{0}\left(1,z'\right)\mathrm{d}z'+c_{5}\varepsilon\nonumber \\
 & \le c_{4}t^{-1+\hat{\theta}/\alpha}\int_{\left\{ |z'|\ge h^{-1/\alpha}/n_{0}\right\} }\varrho_{\alpha}^{0}\left(1,z'\right)\mathrm{d}z'+c_{5}\varepsilon.\label{esti4: J-phi}
\end{align}
Combining (\ref{esti1: J-phi}), (\ref{esti11: J-phi}), and (\ref{esti4: J-phi})
yields
\[
\lim_{h\downarrow0}h^{-1}\int_{t}^{t+h}\left|J(t+h,s,x)-\varphi(t,x)\right|\mathrm{d}s\le c_{5}\varepsilon.
\]
Since $\varepsilon>0$ is arbitrary, the convergence in (\ref{esti0: J-phi})
follows.

``$Step$ 2'': We evaluate the integral $\int_{0}^{t}\partial_{t}J(t^{*},s,x)\mathrm{d}s$.
If $t>s$, then
\begin{align}
\partial_{t}J(t,s,x) & =\int_{\Rd}\partial_{t}q(t-s,x,z)\varphi(s,z)\mathrm{d}z\nonumber \\
 & =\int_{\Rd}\mathcal{A}^{z}q(t-s,\cdot,z)(x)\varphi(s,z)\mathrm{d}z\label{eq: prop 3, **}\\
 & =\int_{\Rd}\left(\mathcal{A}^{z}-\mathcal{A}\right)q(t-s,\cdot,z)(x)\varphi(s,z)\mathrm{d}z\nonumber \\
 & \qquad+\int_{\Rd}\mathcal{A}q(t-s,\cdot,z)(x)\varphi(s,z)\mathrm{d}z\nonumber \\
 & =:I_{3}+I_{4}.\label{eq: defi, partial_t J}
\end{align}
For $III$, by (\ref{esti:f_t,integral_differ}), (\ref{esti:f_t,integral_differ, a=00003D1})
and (\ref{condition1:f(t,x)}), we have
\begin{align}
|I_{3}| & \le c_{6}s^{-1+\hat{\theta}/\alpha}\left(1+\ln\left[\left(t-s\right)^{-1}\right]\right)\int_{\Rd}\varrho_{0}^{\hat{\theta}}(t-s,z-x)\mathrm{d}z\nonumber \\
 & \overset{(\ref{esti1:rho})}{\le}c_{7}s^{-1+\hat{\theta}/\alpha}\left(1+\ln\left[\left(t-s\right)^{-1}\right]\right)(t-s)^{-1+\hat{\theta}/\alpha}.\label{esti: partial t J III}
\end{align}
The term $I_{4}$ has already been treated in Proposition \ref{prop1:layerpotential},
see (\ref{eq: AJ=00003Dint Aq ds}) and (\ref{esti2: AJ(t,s,x)}).
Altogether we obtain
\begin{equation}
|\partial_{t}J(t,s,x)|\le c_{8}s^{-1+\gamma/\alpha}\left(1+\ln\left[\left(t-s\right)^{-1}\right]\right)(t-s)^{-1+\left(\hat{\theta}-\gamma\right)/\alpha}.\label{esti: partial t J}
\end{equation}

Consider
\[
H:=\int_{0}^{t}J_{1}(t^{*},s,x)\mathrm{d}s-\int_{0}^{t}J_{1}(t,s,x)\mathrm{d}s.
\]
Note that for $0<s<t$ and $t^{*}\in[t,t+h]$, it holds that
\begin{equation}
|J_{1}(t^{*},s,x)-J_{1}(t,s,x)|\overset{(\ref{esti: partial t J})}{\le}2c_{8}s^{-1+\gamma/\alpha}\left(1+\ln\left(\left(t-s\right)^{-1}\right)\right)(t-s)^{-1+\left(\hat{\theta}-\gamma\right)/\alpha}.\label{esti: J_1(t*)-J_1(t)}
\end{equation}
Since for $s<t\le t^{*}\le t+h$, $\lim_{h\to0}J_{1}(t^{*},s,x)=J_{1}(t,s,x)$,
by (\ref{esti: J_1(t*)-J_1(t)}) and dominated convergence, we obtain
\[
\lim_{h\to0}\int_{0}^{t}|J_{1}(t^{*},s,x)\mathrm{d}s-J_{1}(t,s,x)|\mathrm{d}s=0.
\]
So we get $\lim_{h\to0}|H|=0$. By (\ref{eq: prop 3, *}), (\ref{esti0: J-phi})
and (\ref{eq: prop 3, **}), we obtain (\ref{eq:partqF^n}).

``$Step$ 3'': To see that the function $t\mapsto\partial_{t}V(t,x)$
is continuous, we can argue as above, namely, for $h\in(0,\delta)$,
\begin{align*}
 & \int_{0}^{t+h}\int_{\Rd}\mathcal{A}^{z}q(t+h-s,\cdot,z)(x)\varphi(s,z)\mathrm{d}z\mathrm{d}s\\
 & \quad=\int_{0}^{t}\int_{\Rd}\mathcal{A}^{z}q(t+h-s,\cdot,z)(x)\varphi(s,z)\mathrm{d}z\mathrm{d}s\\
 & \qquad+\int_{t}^{t+h}\int_{\Rd}\mathcal{A}^{z}q(t+h-s,\cdot,z)(x)\varphi(s,z)\mathrm{d}z\mathrm{d}s,
\end{align*}
where the second term on the right-hand side goes to 0 as $h\to0$,
since by (\ref{esti: partial t J}),
\begin{align*}
 & \lim_{h\to0}\int_{t}^{t+h}|J_{1}(t+h,s,x)|\mathrm{d}s\\
 & \quad\le\lim_{h\to0}\int_{t}^{t+h}s^{-1+\gamma/\alpha}\left(1+\ln\left(\left(t+h-s\right)^{-1}\right)\right)(t+h-s)^{-1+(\hat{\theta}-\gamma)/\alpha}\mathrm{d}s=0,
\end{align*}
while the first term converges to $\int_{0}^{t}\int_{\Rd}\mathcal{A}^{z}q(t-s,\cdot,z)(x)\varphi(s,z)\mathrm{d}z\mathrm{d}s$
by (\ref{eq: prop 3, **}), (\ref{esti: partial t J}) and dominated
convergence.
\end{proof}
\begin{cor}
\label{cor: contin. of V}Let $\varphi$ and $V$ be as in Proposition
\ref{prop1:layerpotential}. Then the function $(t,x)\mapsto V(t,x)$
is bounded continuous on $(0,1]\times\Rd$.
\end{cor}
\begin{proof}
According to (\ref{eq 1: conti. of V}), the function $V$ is obviously
bounded on $(0,1]\times\Rd$. Let $(t_{0},x_{0})\in(0,1]\times\Rd$
be fixed. Choose $\varepsilon>0$ such that $\varepsilon<t_{0}$.
In view of (\ref{eq: prop 3, **}) and (\ref{esti: partial t J}),
we obtain for $s<t$ and $x\in\Rd$,
\begin{align*}
 & \left|\int_{\Rd}\mathcal{A}^{z}q(t-s,\cdot,z)(x)\varphi(s,z)\mathrm{d}z\right|\\
 & \quad\le c_{1}s^{-1+\gamma/\alpha}\left(1+\ln\left[\left(t-s\right)^{-1}\right]\right)(t-s)^{-1+\left(\hat{\theta}-\gamma\right)/\alpha}.
\end{align*}
Arguing as in (\ref{bound for AV on =00005Be,1=00005D}), we get
\[
\left|\int_{0}^{t}\int_{\Rd}\mathcal{A}^{z}q(t-s,\cdot,z)(x)\varphi(s,z)\mathrm{d}z\mathrm{d}s\right|\le c_{2}\left(1+\ln\left(t^{-1}\right)\right)t^{-1},\quad t\in(0,1],x\in\Rd.
\]
By (\ref{eq:partqF^n}), we see that $\partial_{t}V(t,x)$ is bounded
on $[\varepsilon,1]\times\Rd$. Therefore, for $(t,x)\in[\varepsilon,1]\times\Rd$,
\begin{align}
|V(t,x)-V(t_{0},x_{0})| & \le|V(t,x)-V(t_{0},x)|+|V(t_{0},x)-V(t_{0},x_{0})|\nonumber \\
 & \le c_{3}|t-t_{0}|+|V(t_{0},x)-V(t_{0},x_{0})|.\label{eq2: conti. of V}
\end{align}
By (\ref{eq: differentiability of J in x}), $J(t,s,x)$ is continuous
in $x$. Since $V(t,x)=\int_{0}^{t}J(t,s,x)\mathrm{d}s$, it follows
from (\ref{eq 1: conti. of V}) and dominated convergence that for
each $t\in(0,1]$, the function $x\mapsto V(t,x)$ is continuous.
In view of (\ref{eq2: conti. of V}), we get $\lim_{(t,x)\to(t_{0},x_{0})}V(t,x)=V(t_{0},x_{0})$.
\end{proof}
Next, we show that $p(t,x,y)$ defined in (\ref{defi: p(t,x,y)})
is the fundamental solution to the Cauchy problem of the equation
$\partial_{t}u=\mathcal{A}u.$
\begin{prop}
\label{Prop. : p is funda. solu.}Let $\phi\in C_{0}^{\infty}(\Rd)$.
Define $u(t,x):=\int_{\Rd}p(t,x,y)\phi(y)\mathrm{d}y$, $t\in(0,1]$,
and $u(0,x):=\phi(x)$, where $x\in\Rd$. Then $u\in C_{b}([0,1]\times\Rd)$
and
\begin{equation}
\partial_{t}u(t,x)=\mathcal{A}u(t,\cdot)(x),\quad t\in(0,1],\,x\in\Rd.\label{eq 0: pf fund. sol.}
\end{equation}
Moreover,\textcolor{red}{{} }\textcolor{black}{for each $x\in\Rd$,
$t\mapsto\partial_{t}u(t,x)$ is continuous on $(0,1]$; for each
$t\in(0,1]$, $x\mapsto\partial_{t}u(t,x)$ is continuous on $\Rd$.}
\end{prop}
\begin{proof}
Set
\[
I_{1}(t,x):=\int_{\Rd}q(t,x,y)\phi(y)\mathrm{d}y
\]
and
\begin{align*}
I_{2}(t,x) & :=\int_{0}^{t}\int_{\Rd}\int_{\Rd}q(t-s,x,z)\Phi(s,z,y)\phi(y)\mathrm{\mathrm{d}\mathit{y}d}z\mathrm{d}s\\
 & =\int_{0}^{t}\int_{\Rd}q(t-s,x,z)\varphi(s,z)\mathrm{d}z\mathrm{d}s,
\end{align*}
where $\varphi(s,z):=\int_{\Rd}\Phi(s,z,y)\phi(y)\mathrm{\mathrm{d}\mathit{y}}$.
Then $\varphi$ satisfies (\ref{condition1:f(t,x)}) and (\ref{condition2:f(t,x)}).

By Proposition \ref{prop1:layerpotential}, $\mathcal{A}I_{2}(t,\cdot)(x)$
is well-defined for all $t\in(0,1]$ and $x\in\Rd$, and it holds
that
\begin{align}
\mathcal{A}u(t,\cdot)(x) & =\int_{\Rd}\mathcal{A}q(t,\cdot,y)(x)\phi(y)\mathrm{d}y\nonumber \\
 & \quad+\int_{0}^{t}\int_{\Rd}\mathcal{A}q(t-s,\cdot,z)(x)\varphi(s,z)\mathrm{d}z\mathrm{d}s.\label{eq 1: pf fund. sol.}
\end{align}
For $t\in(0,1]$ and $x\in\Rd$, we have
\begin{equation}
{\color{red}{\color{black}\partial_{t}I_{1}(t,x)=\int_{\Rd}\mathcal{A}^{y}q(t,\cdot,y)(x)\phi(y)\mathrm{d}y,}}\label{eq 2: pf fund. sol.}
\end{equation}
and, by Proposition \ref{prop2:layerpotential},
\begin{equation}
{\color{black}{\color{black}{\color{black}{\color{red}{\color{black}\partial_{t}I_{2}(t,x)=\varphi(t,x)+\int_{0}^{t}\int_{\Rd}\mathcal{A}^{z}q(t-s,\cdot,z)(x)\varphi(s,z)\mathrm{d}z{\color{black}\mathrm{d}}s.}}}}}\label{eq 3: pf fund. sol.}
\end{equation}
So for $t\in(0,1]$ and $x\in\Rd$,
\begin{align}
\varphi(t,x) & =\int_{\Rd}\Phi(t,x,y)\phi(y)\mathrm{d}y\nonumber \\
 & \overset{(\ref{defi Phi})}{=}\int_{\Rd}F(t,x,y)\phi(y)\mathrm{d}y\nonumber \\
 & \qquad+\int_{\Rd}\left(\int_{0}^{t}\int_{\Rd}F(t-s,x,z)\Phi(s,z,y)\mathrm{d}z\mathrm{d}s\right)\phi(y)\mathrm{d}y\nonumber \\
 & =\int_{\Rd}\left(\mathcal{A}-\mathcal{A}^{y}\right)q(t,\cdot,y)(x)\phi(y)\mathrm{d}y\nonumber \\
 & \qquad+\int_{0}^{t}\int_{\Rd}\left(\mathcal{A}-\mathcal{A}^{z}\right)q(t-s,\cdot,z)(x)\varphi(s,z)\mathrm{d}z\mathrm{d}s.\label{eq 4: pf fund. sol.}
\end{align}
Combining (\ref{eq 1: pf fund. sol.}), (\ref{eq 2: pf fund. sol.}),
(\ref{eq 3: pf fund. sol.}) and (\ref{eq 4: pf fund. sol.}), we
arrive at (\ref{eq 0: pf fund. sol.}).

By Corollary \ref{cor: contin. of V}, we see that $u\in C_{b}((0,1]\times\Rd)$.
So it remains to show the continuity of $(t,x)\mapsto u(t,x)$ at
$(t,x)=(0,x_{0}),$ where $x_{0}\in\Rd$. We have
\begin{align*}
|u(t,x)-u(0,x_{0})| & \le|u(t,x)-u(0,x)|+|\phi(x)-\phi(x_{0})|.
\end{align*}
So it suffices to show that $\lim_{t\to0}u(t,x)=u(0,x)$, and the
convergence is uniform with respect to $x\in\Rd$. Noting that $|\phi(y)-\phi(x)|\le c_{1}(1\wedge|x-y|^{\alpha/2})$,
we obtain
\begin{align*}
|I_{1}(t,x)-\phi(x)| & \le\left|\int_{\Rd}q(t,x,y)\left[\phi(y)-\phi(x)\right]\mathrm{d}y\right|\\
 & \qquad+\left|\int_{\Rd}q(t,x,y)\phi(x)\mathrm{d}y-\phi(x)\right|\\
 & \le c_{2}\int_{\Rd}\varrho_{\alpha}^{\alpha/2}(t,y-x)\mathrm{d}y+\left|\phi(x)\int_{\Rd}\left[f_{t}^{y}(y-x)-f_{t}^{x}(y-x)\right]\mathrm{d}y\right|\\
 & \overset{(\ref{esti:f_t_differof_y_y'})}{\le}c_{3}t^{1/2}+c_{4}t^{\theta/\alpha},
\end{align*}
which shows that $\lim_{t\to0}\sup_{x\in\Rd}|I_{1}(t,x)-\phi(x)|=0$.
Finally, it follows from (\ref{eq 1: conti. of V}) that $\lim_{t\to0}\sup_{x\in\Rd}|I_{2}(t,x)|=0$.
So $u(t,x)\to u(0,x)$ uniformly in $x\in\Rd$ as $t\to0$.

Since $\partial_{t}u(t,x)=\partial_{t}I_{1}(t,x)+\partial_{t}I_{2}(t,x)$,
the continuity of \textcolor{black}{$t\mapsto\partial_{t}u(t,x)$
follows easily by (}\ref{eq 2: pf fund. sol.}), (\ref{eq 3: pf fund. sol.})
and\textcolor{black}{{} Proposition }\ref{prop2:layerpotential}. \textcolor{black}{Noting
that $x\mapsto\mathcal{A}^{y}q(t,\cdot,y)(x)$ is continuous and}
\textcolor{black}{for $|x-x_{0}|\le t^{1/\alpha}$,}
\begin{align*}
\lvert\mathcal{A}^{y}q(t,\cdot,z)(x)\rvert & \le c_{5}\varrho_{0}^{0}(t,x-z)\overset{(\ref{ineq 1: chen})}{\le}c_{6}\varrho_{0}^{0}(t,x_{0}-z),
\end{align*}
\textcolor{black}{the continuity of $x\mapsto\partial_{t}I_{1}(t,x)$
follows by (}\ref{eq 2: pf fund. sol.}) and dominated convergence.
Similarly, \textcolor{black}{$x\mapsto\partial_{t}I_{2}(t,x)$ is
also continuous. So }the continuity of \textcolor{black}{$x\mapsto\partial_{t}u(t,x)$
follows. }This completes the proof.
\end{proof}
\begin{prop}
\label{prop: p is the density}Let $\left(X,\left(\mathbf{{P}}^{x}\right)\right)$
be the Markov process associated with the operator $\mathcal{A}$
defined in \emph{(\ref{defi 2: A})}. Then the function $p(t,x,y)$,
$(t,x,y)\in(0,1]\times\mathbb{R}^{2d}$, is the transition density
of $\left(X,\left(\mathbf{{P}}^{x}\right)\right)$, namely, for each
$0<t\le1$ and $x\in\Rd$,
\[
\mathbf{{P}}^{x}\left(X_{t}\in E\right)=\int_{E}p(t,x,y)\mathrm{d}y,\quad\forall E\in\mathcal{B}(\Rd).
\]

\end{prop}
\begin{proof}
Let $0<t\le1$ be fixed. Consider $\phi\in C_{0}^{\infty}(\Rd)$ that
is arbitrary. Define $u(s,x):=\int_{\Rd}p(s,x,y)\phi(y)\mathrm{d}y,\ s>0,x\in\Rd$,
and $u(0,\cdot)=\phi$. Let
\[
\tilde{u}(s,x):=u(t-s,x),\quad0\le s\le t,\,x\in\Rd.
\]
By Theorem \ref{Prop. : p is funda. solu.}, $\tilde{u}\in C_{b}([0,t]\times\Rd)$
and
\begin{equation}
\partial_{s}\tilde{u}(s,x)+\mathcal{A}\tilde{u}(s,x)=0,\quad0\le s<t,\,x\in\Rd,\quad\tilde{u}(t,x)=\phi(x).\label{eq: for u tilde}
\end{equation}

Let $\left(\rho_{n}\right)_{n\in\mathbb{N}}$ be a mollifying sequence
in $\Rd$. Set
\[
\tilde{u}_{n}(s,\cdot):=\tilde{u}(s,\cdot)\ast\rho_{n}.
\]
Then for $0<\varepsilon<t$, we have $\tilde{u}_{n}\in C_{b}^{1,2}([0,t-\varepsilon]\times\Rd)$.
Indeed, for $(s,x)\in[0,t-\varepsilon]\times\Rd$,
\[
\partial_{s}\tilde{u}_{n}(s,x)=\int_{\Rd}\partial_{s}\tilde{u}(s,x-y)\rho_{n}(y)\mathrm{d}y.
\]
Note that for each $x\in\Rd$, $s\mapsto\partial_{s}\tilde{u}(s,x)$
is continuous, which implies that for each $x\in\Rd$, $s\mapsto\partial_{s}\tilde{u}_{n}(s,x)$
is continuous. Since, by (\ref{esti: prop 2, AV}), (\ref{eq 1: pf fund. sol.})
and (\ref{eq: for u tilde}), $\partial_{s}\tilde{u}(s,x)$ is bounded
on $[0,t-\varepsilon]\times\Rd$, it follows that $\partial_{s}\tilde{u}_{n}(s,x)$
is Lipschitz in $x$, uniformly with respect to $s\in[0,t-\varepsilon]$.
Similarly to Corollary \ref{cor: contin. of V}, we conclude that
$\partial_{s}\tilde{u}_{n}\in C_{b}([0,t-\varepsilon]\times\Rd)$.
It is obvious that
\begin{align*}
\partial_{i}\tilde{u}_{n}(s,x) & =\int_{\Rd}\partial_{i}\rho_{n}(x-y)\tilde{u}(s,y)\mathrm{d}y\\
 & =\int_{\Rd}\tilde{u}(s,x-y)\partial_{i}\rho_{n}(y)\mathrm{d}y\in C_{b}([0,t-\varepsilon]\times\Rd).
\end{align*}
The cases for second order derivatives are similar. So $\tilde{u}_{n}\in C_{b}^{1,2}([0,t-\varepsilon]\times\Rd)$.

According to \cite[Theorem (1.1)]{MR0433614}, the process
\[
\tilde{u}_{n}(s,X_{s})-\int_{0}^{s}(\partial_{r}+\mathcal{A})\tilde{u}_{n}(r,X_{r})\mathrm{d}r,\quad s\in[0,t-\varepsilon],
\]
is a $\mathbf{{P}}^{x}$-martingale. So
\[
\mathbf{{E}}^{x}[\tilde{u}_{n}(t-\varepsilon,X_{t-\varepsilon})]-\mathbf{{E}}^{x}[\tilde{u}_{n}(0,X_{0})]=\mathbf{{E}}^{x}\left[\int_{0}^{t-\varepsilon}(\partial_{r}+\mathcal{A})\tilde{u}_{n}(r,X_{r})\mathrm{d}r\right].
\]
As $n\to\infty$, it is clear that $\partial_{s}\tilde{u}_{n}(s,x)\to\partial_{s}\tilde{u}(s,x)$,
since for each $s\in[0,t-\varepsilon]$, $x\mapsto\partial_{s}\tilde{u}(s,x)$
is continuous; moreover, according to (\ref{bound for AV on =00005Be,1=00005D}),
\begin{align*}
\mathcal{A}\tilde{u}_{n}(s,x) & =\mathcal{A}\left(\int_{\Rd}\tilde{u}(s,x-y)\rho_{n}(y)\mathrm{d}y\right)\\
 & =\int_{\Rd}\mathcal{A}\tilde{u}(s,\cdot-y)(x)\rho_{n}(y)\mathrm{d}y\to\mathcal{A}\tilde{u}(s,x),
\end{align*}
where we used the fact that \textcolor{black}{for each $s\in[0,t-\varepsilon]$,
$x\mapsto\mathcal{A}\tilde{u}(s,\cdot)(x)$ is continuous}. So $(\partial_{r}+\mathcal{A})\tilde{u}_{n}(r,X_{r})$
converges boundedly and pointwise to $(\partial_{r}+\mathcal{A})\tilde{u}(r,X_{r})$.
By dominated convergence, we obtain
\[
\mathbf{{E}}^{x}[\tilde{u}(t-\varepsilon,X_{t-\varepsilon})]-\mathbf{{E}}^{x}[\tilde{u}(0,X_{0})]=\mathbf{{E}}^{x}\left[\int_{0}^{t-\varepsilon}(\partial_{r}+\mathcal{A})\tilde{u}(r,X_{r})\mathrm{d}r\right]=0.
\]
So
\[
\mathbf{{E}}^{x}[u(\varepsilon,X_{t-\varepsilon})]=\tilde{u}(0,x)=u(t,x).
\]
Letting $\varepsilon\to0$, we get
\[
u(t,x)=\mathbf{{E}}^{x}[u(0,X_{t-})]=\mathbf{{E}}^{x}[u(0,X_{t})]=\mathbf{{E}}^{x}[\phi(X_{t})],
\]
at least for $t\in I$:=$\left\{ t\in(0,1]:X_{t-}=X_{t},\ \mathbf{{P}}^{x}\mbox{-a.s.}\right\} $.
By \cite[Chap. 3, Lemma 7.7]{MR838085}, the set $(0,1]\setminus I$
is at most countable. Then by the right continuity of $t\mapsto X_{t}$
and the continuity of $t\mapsto u(t,x)$, we obtain for all $t\in(0,1]$,
\[
\mathbf{{E}}^{x}[\phi(X_{t})]=u(t,x)=\int_{\Rd}p(t,x,y)\phi(y)\mathrm{d}y,\quad\forall\phi\in C_{0}^{\infty}(\Rd).
\]
This means that $p(t,x,\cdot)$ is the density function of the distribution
of $X_{t}$ under $\mathbf{{P}}^{x}$.
\end{proof}
The next proposition is about a gradient estimate on $p(t,x,y)$ for
the case $1<\alpha<2$.
\begin{prop}
\label{prop: grad on p, t <1}Suppose that $1<\alpha<2$. Then there
exists a constant $C_{31}=C_{31}(d,\alpha,\kappa_{0},\kappa_{1},\kappa_{2},\theta)>0$
such that for all $(t,x,y)\in(0,1]\times\Rd\times\Rd$,
\[
|\nabla_{x}p(t,x,y)|\le C_{31}t^{1-1/\alpha}\left(t^{1/\alpha}+|x|\right)^{-d-\alpha}.
\]
\end{prop}
\begin{proof}
Recall that $p=q+q\otimes\Phi.$ By (\ref{upperbound:gradientf_t})
and Remark \ref{rem: f_t^y}, we obtain
\begin{equation}
|\nabla_{x}q(t,x,y)|\le c_{1}\varrho_{\alpha-1}^{0}(t,x-y),\quad(t,x,y)\in(0,1]\times\Rd\times\Rd.\label{esti1:gradp}
\end{equation}
Since
\[
\nabla_{x}\left(q\otimes\Phi(t,x,y)\right)=\int_{0}^{t}\int_{\Rd}\nabla_{x}q(t-s,x,z)\Phi(s,z,y)\mathrm{d}z\mathrm{d}s,
\]
we get that for $(t,x,y)\in(0,1]\times\Rd\times\Rd$,
\begin{align}
 & |\nabla_{x}\left(q\otimes\Phi(t,x,y)\right)|\nonumber \\
 & \quad\le c_{2}\int_{0}^{t}\int_{\Rd}\varrho_{\alpha-1}^{0}(t-s,x-z)\left\{ \varrho_{\hat{\theta}}^{0}(s,z-y)+\varrho_{0}^{\hat{\theta}}(s,z-y)\right\} \mathrm{d}z\mathrm{d}s\nonumber \\
 & \quad\overset{(\ref{esti3:rho})}{\le}c_{3}\varrho_{\hat{\theta}+\alpha-1}^{0}(t,x,y)+c_{4}\varrho_{\alpha-1}^{\hat{\theta}}(t,x,y)\le c_{5}\varrho_{\alpha-1}^{0}(t,x,y).\label{esti2:gradp}
\end{align}
Now, the assertion follows by (\ref{esti1:gradp}) and (\ref{esti2:gradp}).
\end{proof}
We conclude this section with the following theorem.
\begin{prop}
Consider the operator $\mathcal{A}$ \emph{given in (\ref{defi 2: A})}
and assume that $n(\cdot,\cdot)$ satisfies Assumption \ref{ass1:the function n}.
Then for the Markov process $\left(X,\left(\mathbf{{P}}^{x}\right)\right)$
associated with $\mathcal{A}$, there exists a jointly continuous
transition density $p(t,x,y)$ such that for all $t>0$, $x\in\Rd$
and $E\in\mathcal{B}(\Rd)$,
\[
\mathbf{{P}}^{x}\left(X_{t}\in E\right)=\int_{E}p(t,x,y)\mathrm{d}y.
\]
Moreover, for each $T>0,$ there exists a constant $C_{32}=C_{32}(d,\alpha,\kappa_{0},\kappa_{1},\kappa_{2},\theta,T)>0$
such that
\begin{equation}
p(t,x,y)\le C_{32}t\left(t^{1/\alpha}+|x-y|\right)^{-d-\alpha},\quad x,y\in\Rd,\ 0<t\le T.\label{Theorem 2, eq 1}
\end{equation}
For the case $1<\alpha<2$, there exists also a constant $C_{33}=C_{33}(d,\alpha,\kappa_{0},\kappa_{1},\kappa_{2},\theta,T)>0$
such that
\begin{equation}
|\nabla_{x}p(t,x,y)|\le C_{33}t^{1-1/\alpha}\left(t^{1/\alpha}+|x-y|\right)^{-d-\alpha},\quad x,y\in\Rd,\ 0<t\le T.\label{Theorem 2, eq 2}
\end{equation}
\label{theorem Consider-the-operator}
\end{prop}
\begin{proof}
Let $T>0$ be fixed and set $a:=T^{-1/\alpha}$. Define $\mathbf{\tilde{P}}^{x}=\mathbf{P}^{x/a}$
and $Y_{t}:=aX_{a^{-\alpha}t}$, $t\ge0.$ By Lemma \ref{claim:(Rescaling)},
Remark \ref{rem: invariance of holder}, and Propositions \ref{Prop. esti: p(t,x,y) }
and \ref{prop: p is the density}, the Markov process $\left(Y,\left(\mathbf{{\tilde{P}}}^{x}\right)\right)$
has a jointly continuous transition density $\tilde{p}(t,x,y)$, $(t,x,y)\in(0,1]\times\mathbb{R}^{2d}$.
Moreover, there exists a constant $c_{1}=c_{1}(d,\alpha,\kappa_{0},\kappa_{1},\kappa_{2},\theta)>0$
such that
\begin{equation}
\tilde{p}(t,x,y)\overset{(\ref{main esti: p})}{\le}c_{1}t\left(t^{1/\alpha}+|x-y|\right)^{-d-\alpha},\quad t\in(0,1],\,x,y\in\Rd.\label{esti: ptilde}
\end{equation}
It follows that for each $t\in(0,T]$ and $x\in\Rd$, the law of $X_{t}$
under $\mathbf{{P}}^{x}$ is absolutely continuous with respect to
the Lebesgue measure and thus has a density function $p(t,x,\cdot)$.
Since

\[
\tilde{p}(t,x,y)\mathrm{d}y=\mathbf{{\tilde{P}}}^{x}(Y_{t}\in\mathrm{d}y)=\mathbf{{P}}^{x/a}(aX_{a^{-\alpha}t}\in\mathrm{d}y)=a^{-d}p\left(a^{-\alpha}t,x/a,y/a\right)\mathrm{d}y,
\]
we obtain
\begin{align*}
p(t,x,y)=a^{d}\tilde{p}(a^{\alpha}t,ax,ay) & \overset{(\ref{esti: ptilde})}{\le}c_{1}a^{d}a^{\alpha}t\left((a^{\alpha}t)^{1/\alpha}+|ax-ay|\right)^{-d-\alpha}\\
 & =c_{1}t\left(t^{1/\alpha}+|x-y|\right)^{-d-\alpha},\quad\forall x,y\in\Rd,\ 0<t\le T.
\end{align*}
Moreover, by the continuity of $\tilde{p}(t,x,y)$, the function $(t,x,y)\mapsto p(t,x,y)$
is continuous on $(0,T]\times\Rd\times\Rd$. In view of Proposition
\ref{prop: grad on p, t <1}, the estimate (\ref{Theorem 2, eq 2})
can be similarly proved. This completes the proof.
\end{proof}
\begin{rem}
\label{rem: p(t,x,y)}Let $p(t,x,y)$ be as in Proposition \ref{theorem Consider-the-operator}.
It follows from (\ref{defi: p(t,x,y)}), (\ref{esti: q tensor Phi})
and Lemma \ref{lem: lower esti for f_t} that there exist $t_{0}=t_{0}(d,\alpha,\kappa_{0},\kappa_{1},\kappa_{2},\theta)\in(0,1)$
and $C_{34}=C_{34}(d,\alpha,\kappa_{0},\kappa_{1},\kappa_{2},\theta)>0$
such that
\begin{equation}
p(t,x,y)\ge C_{34}t^{-d/\alpha},\quad\forall t\in(0,t_{0}],\ |x\text{\textminus}y|\le t^{1/\alpha}.\label{esti: lower bound p(t,x,y)}
\end{equation}
\end{rem}

\section{Transition density of the Markov process associated with $\mathcal{L}$}

In this section we assume $1<\alpha<2$. In this case, we still need
to handle the extra term $b(x)\cdot\nabla f(x)$ in the definition
of $\mathcal{L}f$. Throughout this section we assume Assumptions
\ref{ass1:the function n} and \ref{ass2:The-vector-field} are true.

Let $p(t,x,y)$ be as in Proposition \ref{theorem Consider-the-operator}.
It follows from the continuity of $p(t,x,y)$ and the Markov property
that
\begin{equation}
\int_{\Rd}p(s,x,z)p(t,z,y)\mathrm{d}z=p(t+s,x,y),\quad t,s>0,\ x,y\in\Rd.\label{eq: CK}
\end{equation}
By (\ref{eq: CK}) and Theorem \ref{theorem Consider-the-operator},
there exists a constant $C_{35}=C_{35}(d,\alpha,\kappa_{0},\kappa_{1},\kappa_{2},\theta)>0$
such that for all $t>0$ and $x,y\in\Rd$,
\begin{equation}
p(t,x,y)\le C_{35}e^{C_{35}t}t\left(t^{1/\alpha}+|x-y|\right)^{-d-\alpha}\label{bound: global pb}
\end{equation}
and
\begin{equation}
|\nabla_{x}p(t,x,y)|\le C_{35}e^{C_{35}t}t^{1-1/\alpha}\left(t^{1/\alpha}+|x-y|\right)^{-d-\alpha}.\label{bound: global gradient pb}
\end{equation}

For $t>0$ and $x,y\in\Rd,$ let $l_{0}(t,x,y):=p(t,x,y)$. Then
\begin{align*}
 & \int_{0}^{t}\int_{\Rd}\left|l_{0}(t-s,x,z)b(z)\cdot\nabla_{z}p(s,z,y)\right|\mathrm{d}z\mathrm{d}s\\
 & \quad\le\kappa_{3}C_{35}^{2}e^{C_{35}t}\int_{0}^{t}\int_{\Rd}\varrho_{\alpha}^{0}(t-s,x-z)\varrho_{\alpha-1}^{0}(s,z-y)\mathrm{d}z\mathrm{d}s\\
 & \quad\le\kappa_{3}C_{7}C_{35}^{2}e^{C_{35}t}\mathcal{B}\left(1,1-\alpha^{-1}\right)\varrho_{2\alpha-1}^{0}(t,x,y).
\end{align*}
So
\begin{equation}
l_{1}(t,x,y):=\int_{0}^{t}\int_{\Rd}l_{0}(t-s,x,z)b(z)\cdot\nabla_{z}p(s,z,y)\mathrm{d}z\mathrm{d}s,\quad t>0,\ x,y\in\Rd,\label{defi: l_1}
\end{equation}
is well-defined. Similarly, we can define recursively
\begin{equation}
l_{n}(t,x,y):=\int_{0}^{t}\int_{\Rd}l_{n-1}(t-s,x,z)b(z)\cdot\nabla_{z}p(s,z,y)\mathrm{d}z\mathrm{d}s,\quad t>0,\ x,y\in\Rd.\label{defi: l_n}
\end{equation}
By induction, we easily get that for $t>0$ and $x,y\in\Rd$,
\begin{align}
 & \left|l_{n}(t,x,y)\right|\nonumber \\
 & \le C_{35}\left(\kappa_{3}C_{7}C_{35}\right)^{n}e^{C_{35}t}\prod_{i=1}^{n}\mathcal{B}\left(\frac{\alpha+\left(i-1\right)(\alpha-1)}{\alpha},\frac{\alpha-1}{\alpha}\right)\varrho_{\alpha+n(\alpha-1)}^{0}(t,x,y)\nonumber \\
 & =\frac{C_{35}\left(\kappa_{3}C_{7}C_{35}\Gamma\left(1-\alpha^{-1}\right)\right)^{n}e^{C_{35}t}}{\Gamma\left(1+n\left(1-\alpha^{-1}\right)\right)}\varrho_{\alpha+n(\alpha-1)}^{0}(t,x,y)\label{esti1: l_n}
\end{align}
and
\begin{align}
 & \left|\nabla_{x}l_{n}(t,x,y)\right|\nonumber \\
 & \le C_{35}\left(\kappa_{3}C_{7}C_{35}\right)^{n}e^{C_{35}t}\prod_{i=1}^{n}\mathcal{B}\left(\frac{i(\alpha-1)}{\alpha},\frac{\alpha-1}{\alpha}\right)\varrho_{(n+1)(\alpha-1)}^{0}(t,x,y)\nonumber \\
 & =\frac{C_{35}\left(\kappa_{3}C_{7}C_{35}\right)^{n}\left(\Gamma\left(1-\alpha^{-1}\right)\right)^{n+1}e^{C_{35}t}}{\Gamma\left(\left(1+n\right)\left(1-\alpha^{-1}\right)\right)}\varrho_{(n+1)(\alpha-1)}^{0}(t,x,y).\label{esti1: l_n-1}
\end{align}

\begin{rem}
Similarly as above, for $(t,x,y)\in(0,\infty)\times\Rd\times\Rd,$
 define $|l|_{0}(t,x,y):=p(t,x,y)$ and then recursively
\[
|l|_{n}(t,x,y):=\int_{0}^{t}\int_{\Rd}|l|_{n-1}(t-s,x,z)|b(z)|\cdot|\nabla_{z}p(s,z,y)\mathrm{|d}z\mathrm{d}s.
\]
In view of Lemma \ref{Lemma uk_lambda}, we can follow the same argument
as in \cite[p.~191, (40)]{MR2283957} to obtain the existence of $\lambda_{0}>0$
and $C_{36}=C_{36}(d,\alpha,\kappa_{0},\kappa_{1},\kappa_{2},\theta,\kappa_{3})>0$
such that
\begin{equation}
\sum_{n=0}^{\infty}\int_{0}^{\infty}e^{-\lambda t}|l|_{n}(t,x,y)\mathrm{d}t\le C_{36}u_{\lambda}(x-y),\quad\forall\lambda>\lambda_{0},\ x,y\in\Rd,\label{esti: |l|_n}
\end{equation}
where $u_{\lambda}$ is defined in Sect. \ref{subsec:Some-inequalities-and}.
\end{rem}
\begin{prop}
\label{Prop: existence of l(t,x,y)}Assume $1<\alpha<2$. Let $\mathcal{L}$
and $\left(X,\left(\mathbf{{L}}^{x}\right)\right)$ be as in Theorem
\ref{thm: main}, and $l_{n}$ be as in \emph{(\ref{defi: l_n})}.
Then $\left(X,\left(\mathbf{{L}}^{x}\right)\right)$ has a jointly
continuous transition density $l(t,x,y)$ given by
\begin{equation}
l(t,x,y):=\sum_{n=0}^{\infty}l_{n}(t,x,y),\quad(t,x,y)\in(0,\infty)\times\Rd\times\Rd,\label{defi: l}
\end{equation}
where the series on the right-hand side of \emph{(\ref{defi: l})}
converges locally uniformly on $(0,\infty)\times\Rd\times\Rd$. Moreover,
it holds that for all $(t,x,y)\in(0,\infty)\times\Rd\times\Rd$,
\begin{equation}
l(t,x,y)=p(t,x,y)+\int_{0}^{t}\int_{\mathbb{R}^{d}}l(\tau,x,z)b(z)\cdot\nabla_{z}p(t-\tau,z,y)\mathrm{d}z\mathrm{d}\tau.\label{neweqduhamel}
\end{equation}
\end{prop}
\begin{proof}
Let $T>1$ be fixed. By (\ref{esti1: l_n}), we get for $t\in(0,T]$
and $x,y\in\Rd$,
\begin{equation}
\left|l_{n}(t,x,y)\right|\le\frac{C_{35}\left(\kappa_{3}C_{7}C_{35}T^{(1-\alpha^{-1})}\Gamma\left(1-\alpha^{-1}\right)\right)^{n}e^{C_{35}T}}{\Gamma\left(1+n\left(1-\alpha^{-1}\right)\right)}\varrho_{\alpha}^{0}(t,x,y).\label{esti: l_n}
\end{equation}
The local uniform convergence of $\sum_{n=0}^{\infty}l_{n}(t,x,y)$
follows from (\ref{esti: l_n}). It is also easy to see that (\ref{neweqduhamel})
is true. By induction and a similar argument as in \cite[Lemma 14]{MR2283957},
we see that\textcolor{red}{{} }\textcolor{black}{$l_{n}(t,x,y)$ is
jointly continuous in $(t,x,y)\in(0,\infty)\times\Rd\times\Rd$, which,
together with the local uniform convergence, implies the joint continuity
of $l(t,x,y)$. }

For $\lambda>C_{35}\vee\lambda_{0}$ and $f\in\mathcal{B}_{b}(\Rd)$,
define
\[
R^{\lambda}f(x):=\int_{0}^{\infty}\int_{\mathbb{R}^{d}}e^{-\lambda t}p(t,x,y)f(y)\mathrm{d}y\mathrm{d}t,\quad x\in\Rd,
\]
and
\begin{equation}
S^{\lambda}f(x):=\int_{0}^{\infty}\int_{\mathbb{R}^{d}}e^{-\lambda t}l(t,x,y)f(y)\mathrm{d}y\mathrm{d}t,\quad x\in\Rd.\label{defi: S^lambda}
\end{equation}
Note that $S^{\lambda}$ in (\ref{defi: S^lambda}) is well-defined
by (\ref{esti: |l|_n}). If $f$ is bounded measurable, then
\begin{align*}
 & S^{\lambda}f(x)-R^{\lambda}f(x)\\
 & \quad\overset{(\ref{neweqduhamel})}{=}\int_{0}^{\infty}\int_{\mathbb{R}^{d}}e^{-\lambda t}f(y)\left(\int_{0}^{t}\int_{\mathbb{R}^{d}}l(\tau,x,z)b(z)\cdot\nabla_{z}p(t-\tau,z,y)\mathrm{d}z\mathrm{d}\tau\right)\mathrm{d}y\mathrm{d}t.
\end{align*}
Since (\ref{bound: global pb}), (\ref{bound: global gradient pb})
and (\ref{esti: |l|_n}) hold, we can apply Fubini's theorem to get
\begin{align*}
 & S^{\lambda}f(x)-R^{\lambda}f(x)\\
= & \int_{0}^{\infty}\int_{\mathbb{R}^{d}}e^{-\lambda\tau}l(\tau,x,z)\left[b(z)\cdot\nabla_{z}\left(\int_{\tau}^{\infty}\int_{\mathbb{R}^{d}}e^{-\lambda(t-\tau)}p(t-\tau,z,y)f(y)\mathrm{d}y\mathrm{d}t\right)\right]\mathrm{d}z\mathrm{d}\tau,
\end{align*}
namely,
\begin{equation}
S^{\lambda}f-R^{\lambda}f=S^{\lambda}BR^{\lambda}f,\label{eq: S^lambda-R^lambda}
\end{equation}
where $BR^{\lambda}f:=b\cdot\nabla R^{\lambda}f$. Applying (\ref{eq: S^lambda-R^lambda})
$i$ times, we get
\begin{equation}
S^{\lambda}g=\sum_{k=0}^{i}R^{\lambda}(BR^{\lambda})^{k}g+S^{\lambda}(BR^{\lambda})^{i+1}g,\quad\forall\lambda>C_{35}\vee\lambda_{0},\ g\in\mathcal{B}_{b}(\Rd).\label{eq: S^lambda}
\end{equation}
It follows from (\ref{bound: global gradient pb}) that
\begin{equation}
\left\Vert BR^{\lambda}g\right\Vert \le N_{\lambda}\|b\|\cdot\|g\|\le\kappa_{3}N_{\lambda}\|g\|,\quad g\in\mathcal{B}_{b}(\Rd),\label{ineq: BR^lambda}
\end{equation}
where $N_{\lambda}>0$ is a constant with $N_{\lambda}\downarrow0$
as $\lambda\uparrow\infty$. So we can find $\lambda_{1}>C_{35}\vee\lambda_{0}$
such that $N_{\lambda}<1/\kappa_{3}$ for all $\lambda>\lambda_{1}$.
It follows from (\ref{esti: |l|_n}) and (\ref{ineq: BR^lambda})
that $\lim_{i\to\infty}\|S^{\lambda}(BR^{\lambda})^{i+1}g\|=0$ for
all $\lambda>\lambda_{1}$. Therefore,
\begin{equation}
S^{\lambda}g=\sum_{k=0}^{\infty}R^{\lambda}(BR^{\lambda})^{k}g,\quad\forall\lambda>\lambda_{1},\ g\in\mathcal{B}_{b}(\Rd).\label{series eq: S^lambda}
\end{equation}

Next, we show that $l(t,x,y)$ is the transition density of $\left(X,\left(\mathbf{{L}}^{x}\right)\right)$.
Let $x\in\Rd$ be fixed. For $\lambda>0$ and $f\in\mathcal{B}_{b}(\Rd)$,
define
\[
V^{\lambda}f:=\mathbf{{E}}_{\mathbf{{L}}^{x}}\Big[\int_{0}^{\infty}e^{-\lambda t}f(X_{t})\mathrm{d}t\Big].
\]
For $f\in C_{b}^{2}(\Rd)$, we know that
\begin{align*}
 & f(X_{t})-f(X_{0})-\int_{0}^{t}\mathcal{L}f(X_{u})\mathrm{d}u,\quad t\ge0,
\end{align*}
is a $\mathbf{{L}}^{x}$-martingale. So
\begin{equation}
\mathbf{{E}}_{\mathbf{{L}}^{x}}[f(X_{t})]-f(x)=\mathbf{{E}}_{\mathbf{{L}}^{x}}\Big[\int_{0}^{t}\mathcal{L}f(X_{u})\mathrm{d}u\Big].\label{thmunieq0}
\end{equation}
Multiplying both sides of (\ref{thmunieq0}) by $e^{-\lambda t}$,
integrating with respect to $t$ from $0$ to $\infty$ and then applying
Fubini's theorem, we get for $f\in C_{b}^{2}(\Rd)$,
\begin{align}
 & \mathbf{{E}}_{\mathbf{{L}}^{x}}\Big[\int_{0}^{\infty}e^{-\lambda t}f(X_{t})dt\Big]=\frac{1}{\lambda}f(x)+\frac{1}{\lambda}\mathbf{{E}}_{\mathbf{{L}}^{x}}\Big[\int_{0}^{\infty}e^{-\lambda u}\mathcal{L}f(X_{u})\mathrm{d}u\Big].\label{eq: f in C_b^2}
\end{align}
We now claim
\begin{equation}
\lambda V^{\lambda}f=f(x)+V^{\lambda}(\mathcal{L}f),\quad\forall f\in C^{\alpha+\beta}(\Rd),\label{thmunieq101}
\end{equation}
where $0<\beta<2-\alpha$ (see \cite[Sect.~3.1]{MR1406091} for the
definition of the Hölder space $C^{\alpha+\beta}(\Rd)$). Indeed,
if $f\in C^{\alpha+\beta}(\Rd)$, by convolution with mollifiers,
we can find a sequence $(f_{n})\subset C_{b}^{\infty}(\Rd)$ such
that $f_{n}\to f$ in $C^{\alpha+\beta'}(K)$, for any compact set
$K\subset\Rd$ and $0<\beta'<\beta$. Moreover, $\|f_{n}\|_{C^{\alpha+\beta}(\Rd)}\le\|f\|_{C^{\alpha+\beta}(\Rd)}$.
\textcolor{black}{For details the reader is referred to \cite[p.~438]{MR2945756}.}
Noting that for $|h|\le1,$
\begin{align*}
 & |f_{n}(x+h)-f_{n}(x)-\nabla f_{n}(x)\cdot h|=\left|\int_{0}^{1}\left(\nabla f_{n}(x+rh)-\nabla f_{n}(x)\right)\cdot h\mathrm{d}r\right|\\
 & \quad\le c_{1}\|f_{n}\|_{C^{\alpha+\beta}(\Rd)}|h|^{\alpha+\beta-1}\le c_{1}\|f\|_{C^{\alpha+\beta}(\Rd)}|h|^{\alpha+\beta-1},
\end{align*}
by dominated convergence, we see that $\mathcal{L}f_{n}\to\mathcal{L}f$
boundedly and pointwise as $n\to\infty$. Since (\ref{thmunieq101})
is true for $f_{n}$ by (\ref{eq: f in C_b^2}), the passage to the
limit gives (\ref{thmunieq101}).

Given $g\in C^{\beta}(\Rd)$, it follows from \cite[Proposition 7.4]{MR2555009}
and \cite[Theorem 7.2]{MR2555009} that there exists $f\in C^{\alpha+\beta}(\Rd)$
such that $(\lambda-\mathcal{A})f=g$, where $\lambda>0$. For this
$f$, as in (\ref{thmunieq101}) we have $\lambda R^{\lambda}f=f+R^{\lambda}(\mathcal{A}f)$,
which implies $f=R^{\lambda}g$. Substituting this $f$ in (\ref{thmunieq101}),
we obtain
\begin{equation}
V^{\lambda}g=R^{\lambda}g(x)+V^{\lambda}(BR^{\lambda}g),\quad g\in C^{\beta}(\Rd).\label{neweqvn0}
\end{equation}
After a standard approximation procedure, the equality (\ref{neweqvn0})
holds for any $g\in C_{b}(\Rd)$. Then we can use a monotone class
argument to extend (\ref{neweqvn0}) to all $g\in\mathcal{B}_{b}(\Rd)$.

Similarly to (\ref{eq: S^lambda}), we obtain from (\ref{neweqvn0})
that
\begin{equation}
V^{\lambda}g=\sum_{i=0}^{k}R^{\lambda}(BR^{\lambda})^{i}g(x)+V^{\lambda}(BR^{\lambda})^{k+1}g,\quad g\in\mathcal{B}_{b}(\Rd).\label{eqsnforn}
\end{equation}
For $\lambda>\lambda_{1}$, by (\ref{ineq: BR^lambda}) and the definition
of $\lambda_{1}$, we obtain
\begin{equation}
V^{\lambda}g=\sum_{i=0}^{\infty}R^{\lambda}(BR^{\lambda})^{i}g(x),\quad\forall\lambda>\lambda_{1},\ g\in\mathcal{B}_{b}(\Rd).\label{series eq: V^lambda}
\end{equation}
It follows from (\ref{series eq: S^lambda}) and (\ref{series eq: V^lambda})
that for all $\lambda>\lambda_{1}$ and $g\in\mathcal{B}_{b}(\Rd)$,
\[
\mathbf{{E}}_{\mathbf{{L}}^{x}}\Big[\int_{0}^{\infty}e^{-\lambda t}g(X_{t})\mathrm{d}t\Big]=\int_{0}^{\infty}\int_{\mathbb{R}^{d}}e^{-\lambda t}l(t,x,y)g(y)\mathrm{d}y\mathrm{d}t.
\]
Note that for $g\in C_{b}(\mathbb{R}^{d})$, the function $t\mapsto\int_{\mathbb{R}^{d}}l(t,x,y)g(y)\mathrm{d}y$
is bounded continuous on $(0,T]$ for any $T>0$. By the uniqueness
of the Laplace transform, we obtain
\[
\mathbf{{E}}_{\mathbf{{L}}^{x}}[g(X_{t})]=\int_{\mathbb{R}^{d}}l(t,x,y)g(y)\mathrm{d}y,\quad\forall g\in C_{b}(\mathbb{R}^{d}),\ t>0.
\]
This implies that $l(t,x,\cdot)$ is the density function of the law
of $X_{t}$ under the measure $\mathbf{{L}}^{x}.$
\end{proof}
\begin{rem}
\label{rem: l(t,x,y)}Let $l(t,x,y)$ be as in Proposition \ref{Prop: existence of l(t,x,y)}.
By (\ref{esti: lower bound p(t,x,y)}), (\ref{defi: l}) and (\ref{esti1: l_n}),
there exist $t_{0}=t_{0}(d,\alpha,\kappa_{0},\kappa_{1},\kappa_{2},\theta,\kappa_{3})\in(0,1)$
and $C_{37}=C_{37}(d,\alpha,\kappa_{0},\kappa_{1},\kappa_{2},\theta,\kappa_{3})>0$
such that
\[
l(t,x,y)\ge C_{37}t^{-d/\alpha},\quad\forall t\in(0,t_{0}],\ |x\text{\textminus}y|\le t^{1/\alpha}.
\]
\end{rem}

\section{Proof of Theorem \ref{thm: main}}

Finally, we give\emph{ }the proof of our main result.\emph{ }

\emph{Proof of Theorem \ref{thm: main}.} By Propositions \ref{theorem Consider-the-operator}
and \ref{Prop: existence of l(t,x,y)}, we get the existence of a
jointly continuous transition density $l(t,x,y)$ for $\left(X,\left(\mathbf{{L}}^{x}\right)\right)$.
The claimed upper bounds of $l(t,x,y)$ and $|\nabla_{x}l(t,x,y)|$
follow from (\ref{Theorem 2, eq 1}), (\ref{esti1: l_n}), (\ref{esti1: l_n-1})
and Proposition \ref{Prop: existence of l(t,x,y)}.

We now prove the lower bound of $l(t,x,y)$ by following \cite[Sect.~4.4]{chen2015heat}.
Arguing in the same way as in \cite[p.~306-307]{chen2015heat} (see
also the proof of \cite[Prop.~2.3]{MR1918242}), we conclude that
if $A$ and $B$ are bounded Borel subsets of $\Rd$ with $B$ being
closed and having a positive distance from $A$, then
\begin{equation}
\sum_{s\le t}\mathbf{1}_{A}(X_{s-})\mathbf{1}_{B}(X_{s})-\int_{0}^{t}\mathbf{1}_{A}(X_{s})\left(\int_{B}\frac{n(X_{s},y-X_{s})}{|y-X_{s}|^{d+\alpha}}\mathrm{d}y\right)\mathrm{d}s\label{eq: martingale levy system}
\end{equation}
is a $\mathbf{{L}}^{x}$-martingale.

Let $T>0$ be fixed. By Remarks \ref{rem: p(t,x,y)} and \ref{rem: l(t,x,y)},
there exist constants $t_{0}\in(0,1)$ and $c_{1}>0$ such that
\[
l(t,x,y)\ge c_{1}t^{-d/\alpha},\quad\forall t\in(0,t_{0}],\ |x\text{\textminus}y|\le t^{1/\alpha}.
\]
As in (\ref{eq: CK}), $l(t,x,y)$ satisfies also the Chapman-Kolmogorov's
equation. Iterating $[T/t_{0}]+1$ times, we obtain
\[
l(t,x,y)\ge c_{2}t^{-d/\alpha},\quad\forall t\in(0,T],\ |x\text{\textminus}y|\le3c_{3}t^{1/\alpha},
\]
where $c_{2},c_{3}>0$ are constants. By Lemmas \ref{lem: esti exit time of a ball}
and \ref{lem 2: esti exit time of a ball}, there is a constant $\lambda\in(0,1/2)$
such that for all $t\in(0,T)$ and $x\in\Rd$,
\[
\mathbf{{L}}^{x}\left(\tau_{B_{c_{3}t^{1/\alpha}/2}(x)}\le\lambda t\right)\le\frac{1}{2}.
\]

Below, assume $0<t\le T$ and $|x\text{\textminus}y|>3c_{3}t^{1/\alpha}$.
Set $A_{1}:=B_{c_{3}t^{1/\alpha}}(x)$ and $A_{2}:=B_{c_{3}t^{1/\alpha}/2}(y)$.
Let $\overline{A}_{i}$ the closure of $A_{i}$, $i=1,2$. Similarly
to \cite[p.~309, (4.36)]{chen2015heat}, we have
\[
\mathbf{{L}}^{x}\left(X_{\lambda t}\in B_{c_{3}t^{1/\alpha}}(y)\right)\ge\frac{1}{2}\mathbf{{L}}^{x}\left(X_{\lambda t\wedge\tau_{A_{1}}}\in\overline{A}_{2}\right),
\]
where $\tau_{A_{1}}:=\inf\left\{ t\ge0:\ X_{t}\notin A_{1}\right\} $.
Since
\[
\mathbf{1}_{X_{\lambda t\wedge\tau_{A_{1}}}\in\overline{A}_{2}}=\sum_{s\le\lambda t\wedge\tau_{A_{1}}}\mathbf{1}_{\overline{A}_{1}}(X_{s-})\mathbf{1}_{\overline{A}_{2}}(X_{s}),
\]
by (\ref{eq: martingale levy system}) and optional sampling, we have
\begin{align*}
 & \mathbf{{L}}^{x}\left(X_{\lambda t\wedge\tau_{A_{1}}}\in\overline{A}_{2}\right)\\
 & \quad=\mathbf{{E}}_{\mathbf{{L}}^{x}}\left[\int_{0}^{\lambda t\wedge\tau_{A_{1}}}\mathbf{1}_{\overline{A}_{1}}(X_{s})\left(\int_{\overline{A}_{2}}\frac{n(X_{s},y-X_{s})}{|y-X_{s}|^{d+\alpha}}\mathrm{d}y\right)\mathrm{d}s\right]\\
 & \quad=\mathbf{{E}}_{\mathbf{{L}}^{x}}\left[\int_{0}^{\lambda t\wedge\tau_{A_{1}}}\int_{\overline{A}_{2}}\frac{n(X_{s},y-X_{s})}{|y-X_{s}|^{d+\alpha}}\mathrm{d}y\mathrm{d}s\right].
\end{align*}
The rest of the proof is then the same as in \cite[p.~310]{chen2015heat}.
So we get
\[
l(t,x,y)\ge c_{4}t|x-y|^{-d-\alpha},\quad\forall t\in(0,T],\ |x\text{\textminus}y|>3c_{3}t^{1/\alpha}.
\]
 The theorem is proved. \qed

\bibliographystyle{spmpsci}

\addcontentsline{toc}{section}{\refname}
\begin{thebibliography}{10}


\bibitem{MR958291}
Bass, R.F.: Uniqueness in law for pure jump {M}arkov processes.
\newblock Probab. Theory Related Fields \textbf{79}(2), 271--287 (1988).


\bibitem{MR2555009}
Bass, R.F.: Regularity results for stable-like operators.
\newblock J. Funct. Anal. \textbf{257}(8), 2693--2722 (2009).


\bibitem{MR2095633}
Bass, R.F., Kassmann, M.: Harnack inequalities for non-local operators of
  variable order.
\newblock Trans. Amer. Math. Soc. \textbf{357}(2), 837--850 (2005).


\bibitem{MR1918242}
Bass, R.F., Levin, D.A.: Harnack inequalities for jump processes.
\newblock Potential Anal. \textbf{17}(4), 375--388 (2002).


\bibitem{MR2508568}
Bass, R.F., Tang, H.: The martingale problem for a class of stable-like
  processes.
\newblock Stochastic Process. Appl. \textbf{119}(4), 1144--1167 (2009).

\bibitem{MR0119247}
Blumenthal, R.M., Getoor, R.K.: Some theorems on stable processes.
\newblock Trans. Amer. Math. Soc. \textbf{95}, 263--273 (1960).


\bibitem{MR2283957}
Bogdan, K., Jakubowski, T.: Estimates of heat kernel of fractional {L}aplacian
  perturbed by gradient operators.
\newblock Comm. Math. Phys. \textbf{271}(1), 179--198 (2007)

\bibitem{MR2008600}
Chen, Z.Q., Kumagai, T.: Heat kernel estimates for stable-like processes on
  {$d$}-sets.
\newblock Stochastic Process. Appl. \textbf{108}(1), 27--62 (2003).

\bibitem{chen2015heat}
Chen, Z.Q., Zhang, X.: Heat kernels and analyticity of non-symmetric jump
  diffusion semigroups.
\newblock Probab. Theory Related Fields \textbf{165}(1-2), 267--312 (2016).


\bibitem{chen2016uniqueness}
Chen, Z.Q., Zhang, X.: Uniqueness of stable-like processes.
\newblock arXiv preprint arXiv:1604.02681  (2016)

\bibitem{MR838085}
Ethier, S.N., Kurtz, T.G.: Markov processes: Characterization and convergence.
\newblock Wiley Series in Probability and Mathematical Statistics: Probability
  and Mathematical Statistics. John Wiley \& Sons, Inc., New York (1986).


\bibitem{MR0270403}
Feller, W.: An introduction to probability theory and its applications. {V}ol.
  {II}.
\newblock Second edition. John Wiley \& Sons, Inc., New York-London-Sydney
  (1971)

\bibitem{MR0181836}
Friedman, A.: Partial differential equations of parabolic type.
\newblock Prentice-Hall Inc., Englewood Cliffs, N.J. (1964)

\bibitem{MR3089797}
Kaleta, K., Sztonyk, P.: Upper estimates of transition densities for
  stable-dominated semigroups.
\newblock J. Evol. Equ. \textbf{13}(3), 633--650 (2013).


\bibitem{MR3357585}
Kaleta, K., Sztonyk, P.: Estimates of transition densities and their
  derivatives for jump {L}\'evy processes.
\newblock J. Math. Anal. Appl. \textbf{431}(1), 260--282 (2015).


\bibitem{MR1744782}
Kolokoltsov, V.: Symmetric stable laws and stable-like jump-diffusions.
\newblock Proc. London Math. Soc. (3) \textbf{80}(3), 725--768 (2000).

\bibitem{MR1406091}
Krylov, N.V.: Lectures on elliptic and parabolic equations in {H}\"older
  spaces, \emph{Graduate Studies in Mathematics}, vol.~12.
\newblock American Mathematical Society, Providence, RI (1996).

\bibitem{MR3201992}
Mikulevicius, R., Pragarauskas, H.: On the {C}auchy problem for
  integro-differential operators in {H}\"older classes and the uniqueness of
  the martingale problem.
\newblock Potential Anal. \textbf{40}(4), 539--563 (2014).


\bibitem{MR3145767}
Mikulevi{\v{c}}ius, R., Pragarauskas, H.: On the {C}auchy problem for
  integro-differential operators in {S}obolev classes and the martingale
  problem.
\newblock J. Differential Equations \textbf{256}(4), 1581--1626 (2014).


\bibitem{nolan2016stable}
Nolan, J.: Stable Distributions: Models for Heavy-Tailed Data.
\newblock Springer New York (2016).


\bibitem{MR2945756}
Priola, E.: Pathwise uniqueness for singular {SDE}s driven by stable processes.
\newblock Osaka J. Math. \textbf{49}(2), 421--447 (2012).


\bibitem{MR0433614}
Stroock, D.W.: Diffusion processes associated with {L}\'evy generators.
\newblock Z. Wahrscheinlichkeitstheorie und Verw. Gebiete \textbf{32}(3),
  209--244 (1975)

\bibitem{MR2794975}
Sztonyk, P.: Transition density estimates for jump {L}\'evy processes.
\newblock Stochastic Process. Appl. \textbf{121}(6), 1245--1265 (2011).


\bibitem{MR2286060}
Watanabe, T.: Asymptotic estimates of multi-dimensional stable densities and
  their applications.
\newblock Trans. Amer. Math. Soc. \textbf{359}(6), 2851--2879 (electronic)
  (2007).



\end{thebibliography}

\end{document}